\documentclass[12pt,reqno]{amsart}

\DeclareMathAlphabet{\mathpzc}{OT1}{pzc}{m}{it}

\title[Resonances on products of two rank-one spaces]{Resonances for the Laplacian on Products\\ of two Rank One Riemannian Symmetric Spaces}

\author{J. Hilgert}
\address{Department of Mathematics,
Paderborn University,
Warburger Str. 100,
D-33098 Paderborn,
Germany}
\email{hilgert@math.uni-paderborn.de}
\author{A. Pasquale}
\address{Universit\'e de Lorraine, Institut Elie Cartan de Lorraine, UMR CNRS 7502, Metz, F-57045, France}
\email{angela.pasquale@univ-lorraine.fr}
\author{T. Przebinda}
\address{Department of Mathematics, University of Oklahoma, Norman, OK 73019, USA}
\email{tprzebinda@ou.edu}

\newcommand{\eps}{\varepsilon}

\def\R{\mathbb{R}}
\def\C{\mathbb{C}}

\def\a{\mathfrak a}

\def\c{\mathfrak c}

\def\ss1{\mathfrak s_{\overline 1}}

\def\hs1{\mathfrak h_{\overline 1}}

\def\Dg{\mathrm{D}}
\def\G{\mathrm{G}}

\def\K{\mathrm{K}}
\def\H{\mathrm{H}}
\def\M{\mathrm{M}}

\def\Sg{\mathrm{S}}

\def\Spin{\mathrm{Spin}}

\def\Bbb{\mathbb}

\def\H{\mathrm{H}}

\def\SL{\mathrm{SL}}
\def\SO{\mathrm{SO}}
\def\SU{\mathrm{SU}}
\def\Sp{\mathrm{Sp}}

\def\Ug{\mathrm{U}}

\def\Eg{\mathrm{E}}
\def\Fg{\mathrm{F}}
\def\id{\mathrm{id}}
\def\B{\mathrm{B}}

\def \wt{\widetilde}

\def\bi{\mathbf{i}}

\def\W{\mathsf{W}}

\def\X{\mathsf{X}}

\newcommand{\Rlog}{R_{\rm log}}
\newcommand{\Ulog}{U_{\rm log}}

\def\Zb{\mathbb {Z}}

\def\Re{\mathop{\hbox{\rm Re}}\nolimits}
\def\Im{\mathop{\hbox{\rm Im}}\nolimits}

\newcommand\floor[1]{\lfloor #1\rfloor}

\def\lim{\mathop{\hbox{\rm lim}}\nolimits}

\newcommand\inner[2]{\langle #1,#2\rangle}

\def\Res{\mathop{\hbox{\rm Res}}\limits}

\def\ker{\mathop{\hbox{\rm ker}}\nolimits}

\newcommand{\cHC}{c_{\rm{HC}}}

\newcommand{\cz}{{\mathop{\mathrm c}}}
\newcommand{\sz}{{\mathop{\mathrm s}}}

\def\fonttitre{\textsf}
\newcounter{thh}

\newtheorem{thm}[thh]{\fonttitre{Theorem}}

\newtheorem{pro}[thh]{\fonttitre{Proposition}}
\newtheorem*{pro*}{\fonttitre{Proposition}}
\newtheorem{cor}[thh]{\fonttitre{Corollary}}
\newtheorem*{coro*}{\fonttitre{Corollary}}
\newtheorem{lem}[thh]{\fonttitre{Lemma}}

\theoremstyle{definition} \newtheorem{rem}{\fonttitre{Remark}}

\newtheorem*{defi*}{\fonttitre{D??finition}}

\newtheorem*{nota*}{\fonttitre{Notation}}
\newenvironment{prf}{\begin{proof}}{\end{proof}}
\def\muet{ \ifthenelse{\equal{a}{b}}}
\def\nn{\nonumber}

\begin{document}
\enlargethispage*{1cm}

\date{}

\begin{abstract}
Let $\X=\X_1 \times \X_2$ be a direct product of two rank-one Riemannian symmetric spaces of the noncompact type. We show that when at least one of the two spaces is isomorphic to a real hyperbolic space of odd dimension, the resolvent of the Laplacian of $\X$ can be lifted to a holomorphic function on a Riemann surface which is a branched covering of $\mathbb C$. In all other cases, the resolvent of the Laplacian of $\X$ admits a singular meromorphic lift. The poles of this function are called the resonances of the Laplacian. We determine all resonances and show that the corresponding residue operators are given by convolution with spherical functions parameterized by the resonances. The ranges of these operators are finite dimensional and explicitly realized as direct sums of finite-dimensional irreducible spherical representations of the group of the isometries of $\X$. 
\end{abstract}

\maketitle
\tableofcontents

\section*{Introduction}

In the last thirty years, the problem of the analytic or meromorphic extension of the resolvent of the Laplacian across its continuous spectrum has been studied for several classes of noncompact complete Riemannian manifolds. Examples include asymptotically hyperbolic manifolds, such as the real hyperbolic space and its convex co-compact quotients (\cite{MM87, GZ95, Gui05}), manifolds with asymptotically cylindrical ends (\cite{G89, Mel93}), and Riemannian symmetric or locally symmetric spaces of the noncompact type  (\cite{G89, GZ95, BO00, MW00, MV02, MV05,  Str05, MV07, HP09, HPP14}).  

The extension is obtained by considering the resolvent as an operator acting on a suitable dense subspace of $L^2(\X)$, the Hilbert space of square-integrable functions on the considered Riemannian manifold $\X$, rather than on $L^2(\X)$ itself. 
When this extension turns out to be meromorphic, then its poles are called the \textit{resonances} of the Laplacian of $\X$. These complex numbers plays the role of the eigenvalues for a Laplacian with discrete spectrum and are of interest in physics as they are related to metastable states
(see e.g. \cite{Hunz90, Muck04, CGH06}).

The basic problems are the existence, location, counting estimates and geometric interpretation of the resonances.
These problems are nowadays well understood in the case of Riemannian symmetric spaces of the noncompact type having rank one (see \cite{GZ95, BO00, HP09}). The situation is dramatically different in higher rank, where the existence of resonances has been proven so far only for one particular Riemannian symmetric space of rank two (the space $\SL(3,\R)/\SO(3)$, see \cite{HPP14}). 
The purpose of this paper is to present complete results for the class of rank-two symmetric spaces $\X$ which are direct products $\X_1\times \X_2$ of two Riemannian symmetric spaces
$\X_1$ and  $\X_2$ of the noncompact type and rank one.  These are the simplest cases of higher rank Riemannian symmetric spaces of the noncompact type. Still, they present some interesting new phenomena which we hope will shed some light on the general case. The guiding role of direct products in the understanding of the general Riemannian symmetric case has also been pointed out in \cite{MV02}. Note however that there are no results in \cite{MV02} hinting at the existence or non-existence of resonances on direct products of Riemannian symmetric spaces. 

Before presenting the main results of this paper, we review some background information and give a more precise description of the problems we are looking at. 
Recall that a Riemannian symmetric space of the noncompact type is a homogenous space $\X=\G/\K$, where $\G$ is a noncompact connected real semisimple Lie group with finite center and $\K$ is a maximal compact subgroup of $\G$. The (real) rank of $\X$ is the dimension 
of a Cartan subspace of $\G$: if $\mathfrak{g}=\mathfrak{k}\oplus \mathfrak{p}$ is the Cartan decomposition of the Lie algebra $\mathfrak{g}$ of $\G$, then a Cartan subspace is a maximal abelian subspace $\mathfrak{a}$ of $\mathfrak{p}$. Geometrically, the rank of $\X$ is the maximal dimension of a complete, totally geodesic submanifold of $\X$ that is isometric to a flat Euclidean space $\R^n$ (see e.g. \cite[\S 2.10]{Eb96}). The simplest examples of Riemannian symmetric spaces of the noncompact type are the real hyperbolic spaces $H^n(\R)=\SO_0(1,n)/\SO(n)$. Here $\G=\SO_0(1,n)$ is the generalized Lorentz group, i.e. the connected component of the identity in the group of $(n+1)\times (n+1)$  matrices of determinant $1$ preserving the standard bilinear form on $\R^{n+1}$ of signature $(1,n)$, and $\K=\SO(n)$ is the special orthogonal group of order $n$. They are Riemannian symmetric spaces of rank one. 

The (positive) Laplacian $\Delta$ of a Riemannian symmetric space of the noncompact type $\X$ is an essentially self-adjoint operator on $L^2(\X)$ with continuous spectrum $[\rho_\X^2,+\infty[$, where $\rho_\X^2$ is a positive constant depending on the structure of $\X$. Its resolvent $(\Delta-u)^{-1}$ is therefore a bounded linear operator on $L^2(\X)$ depending holomorphically in $u \in \C \setminus [\rho_\X^2,+\infty[$. 
Denote by $\inner{\cdot}{\cdot}_2$ the inner product on $L^2(\X)$.
Let the resolvent act on the dense subspace $C_c^\infty(\X)$ of $L^2(\X)$ consisting of the smooth compactly-supported functions. 
The problem of extending the resolvent is then to see whether there is some Riemann surface extending $\C \setminus [\rho_\X^2,+\infty[$ to which 
all functions $\inner{(\Delta-u)^{-1}f}{g}_2$, where $f,g\in C_c^\infty(\X)$, admit a meromorphic continuation in $u$.  

As it is customary in the study of resonances, one adds a shift to the Laplacian to bring the 
bottom of its spectrum to the origin of the complex plane and makes the quadratic change of variable $u=z^2$. Let $\C^+=\{z\in \C:\Im z >0\}$ be the upper half-plane and set  
\[
R(z)=(\Delta-\rho_\X^2-z^2)^{-1}\,.
\]
Then the original extension problem is equivalent to meromorphically continue the function 
\[ 
\C^+ \ni z \to \inner{R(z)f}{g}_2 \in \C
\]
across the real axis, for arbitrary $f,g\in C_c^\infty(\X)$.

A natural strategy in this context is to use Fourier analysis. Indeed, the Plancherel theorem for the Helgason-Fourier transform on $\X$ provides an explicit integral formula for 
$\inner{R(z)f}{g}_2$ for all 
$f,g \in L^2(\X)$. Moreover, if $f\in C_c^\infty(\X)$ then, by the Paley-Wiener theorem, $R(z)f$ 
is a smooth function on $\X$.
One obtains an explicit (though complicated) formula for this function as an integral in the real spectral variables; see section \ref{subsection:resolvent kernel}.  

The spectral variables of the Helgason-Fourier transform are the elements $\lambda \in \mathfrak{a}^*_\C$, where $\mathfrak{a}^*_\C$ is the complexified dual of the fixed Cartan subspace $\mathfrak{a}$. The elements of the dual $\mathfrak{a}^*$ of $\mathfrak{a}$ are real spectral variables. The Killing form of $\mathfrak{g}$ endows $\mathfrak{a}^*$ with an inner product $\inner{\cdot}{\cdot}$ which we extend by $\C$-bilinearity to  $\mathfrak{a}^*_\C$. In the spectral variables, the resolvent $R(z)$ becomes the multiplication operator by $(\inner{\lambda}{\lambda}-z^2)^{-1}$, a rational function of $\lambda$. Additional singularities, located along finitely many locally-finite infinite families of parallel affine hyperplanes in  $\mathfrak{a}^*_\C$, arise from the Plancherel density. They determine our meromorphic extension and the residues. The Plancherel density is of the form $[\cHC(\bi\lambda)\cHC(-\bi\lambda)]^{-1}$, where $\cHC$ is Harish-Chandra's $c$-function and $\bi$ denotes the complex multiplication in $\mathfrak{a}_\C^*$ with respect to $\mathfrak{a}^*$. Its explicit expression from Gindikin-Karpelevic's formula plays a major role in all computations; see \eqref{eq:c} and \eqref{eq:Plancherel-density}.

When the Plancherel density has no singularities, the resolvent has a holomorphic extension, either to the entire complex plane, if the rank of $\X$ is odd, or to a logarithmic Riemann surface above $\C$ and branched at the origin $0$, if the rank of $\X$ is even; see \cite[Theorem 3.3(2)]{Str05}. Geometrically, the absence of singularities of the Plancherel density corresponds to the condition that the Lie algebra $\mathfrak{g}$ of $\G$ possesses a unique conjugacy class of Cartan subalgebras. For instance, this happens when $\mathfrak{g}$ has a complex structure and, in rank one, when $\X=H^{2n+1}(\R)$ is a real hyperbolic space of odd dimension. For a product $\X=\X_1\times \X_2$ of rank-one Riemannian symmetric spaces of the noncompact type, this condition corresponds to the case where both $\X_1$ and $\X_2$ are odd-dimensional real hyperbolic spaces.  In the general case, the condition can also be stated in terms of the root system $\Sigma$  of the pair $(\mathfrak{g},\mathfrak{a})$, where $\mathfrak{a}$ is a Cartan subspace of $\mathfrak{g}$. Indeed, $\mathfrak{g}$ has a unique conjugacy class of Cartan subalgebras if and only if $\Sigma$ is reduced and all root multiplicities are even. See
\cite[Chapter IX, Theorem 6.1, and Chapter X, Exercise F4]{He1}.

For all rank-one Riemannian symmetric spaces of the noncompact type $\X$ different from $H^{2n+1}(\R)$,  the Plancherel measure is singular and the extended resolvent has first order poles, i.e. the Laplacian of $\X$ admits resonances in these cases, see e.g. \cite[Theorem 3.8]{HP09} or subsection \ref{subsection:rankone} below. One of the results of this article shows that the singularity of the Plancherel density is not sufficient for the existence of resonances. In Theorem \ref{thm:holoext-even}
we prove that the Laplacian of the product $\X$ of two Riemannian symmetric spaces of rank-one has no resonances when exactly one of the two rank-one factors is an odd dimensional real hyperbolic space $H^{2n+1}(\R)$ for some $n$. In this case, the Plancherel density of $\X$ is singular, but ``not singular enough'' to get resonances.  

For arbitrary Riemannian symmetric spaces of noncompact type, the dichotomy of the extension of the resolvent in the odd and the even rank situations, mentioned above for the case of one conjugacy class of Cartan subalgebras, holds in general. Namely, there is a constant $L>0$  so that the resolvent $R$ has a holomorphic extension across the spectrum of $\X$, either to an open domain of the form $\C\setminus -i[L,+\infty)$, if the rank of $\X$ is odd, or to a logarithmic Riemann surface above $\C\setminus -i[L,+\infty)$ and branched at the origin $0$, if the rank of $\X$ is even. This was proven in \cite[Theorem 3.2 and 3.3]{Str05} and \cite[section 7]{MV05}. 
The constant $L$ depends on the structure of $\X$. It is determined by the condition that 
the Plancherel density $r\to [\cHC(\bi r\omega)\cHC(-\bi r\omega)]^{-1}$ is a holomorphic function of $r \in \C\setminus i((-\infty,-L]\cup [L,+\infty))$ for all $\omega \in \mathfrak{a}^*$
of norm $|\omega|=1$. One can check from \cite[Corollary 2.2]{HP09} that 
\begin{equation}
\label{eq:L}
L=\min\Big\{\frac{1}{2}\big(m_\beta+\frac{m_{\beta/2}}{2}\big) |\beta|; \ \text{$\beta \in \Sigma^+_*$, $m_\beta$ odd} \Big\}\,,
\end{equation}
where $\Sigma^+_*$ is the set of positive unmultipliable roots (i.e. the positive roots $\beta$ so that $2\beta \notin \Sigma$) and $m_\beta$ denotes the multiplicity of the root $\beta$.
Observe that $L=+\infty$ in case all root multiplicities are even. 

The above results show that the region where the possible resonances are located, is the half-line $i(-\infty,-L]$ on the imaginary axis. Understanding the extension of the resolvent across this half-line is the crucial matter. This is the content of sections \ref{section:contour}, 
\ref{section:meroext-odd} and \ref{section:holoext-even} of this paper. For the case where none of the spaces 
$\X_1$ and $\X_2$ is isomorphic to an odd-dimensional real hyperbolic spaces (called below the
\textit{case of two odd multiplicities}), the main result, providing the meromorphic extension of the resolvent across the negative imaginary axis to a suitable Riemann surface as well as the location of the resonances, is Theorem \ref{thm:meroextshiftedLaplacian}. For the case where exactly one of the spaces $\X_1$ or $\X_2$ is isomorphic to an odd-dimensional real hyperbolic spaces (called below the \textit{case of one odd multiplicity}), the main result, providing the holomorphic extension of the resolvent across the negative imaginary axis to a suitable Riemann surface, is Theorem \ref{thm:holoext-even}.  

We refer the reader to the later parts of this paper for the precise statements of these theorems. Here, we point out some features making these results special with respect to those known at present on the resonances of the Laplacian on Riemannian symmetric spaces of the noncompact type.

Recall that resonances are known to exists only when $\X$ is either of rank one and $\neq H^{2n+1}(\R)$, or $\X=\SL(3,\R)/\SO(3)$. In both cases, the Laplacian has an infinite sequence of resonances that are regularly spaced along $i(-\infty,-L]$. The point $-iL$, which by construction is the first radial value on the negative imaginary axis of the singularities of the Plancherel density, is also the first resonance of the Laplacian. Its absolute value squared $L^2$ agrees with the bottom $\rho_\X^2$ of the spectrum of the Laplacian. Furthermore, for $\SL(3,\R)/\SO(3)$, the resonances occur at the branching points of the Riemann surface to which the meromorphic extension of the resolvent across $i(-\infty,-L]$ takes place. On the other hand, for the product $\X=\X_1 \times \X_2$ in the case of two odd multiplicities, the first resonance occurs at $-i\rho_\X$, the point of the negative imaginary axis with absolute value squared equal to $\rho_\X^2$. But $\rho^2_\X=\rho_{\X_1}^2+\rho_{\X_2}^2>L^2=\min\{\rho^2_{\X_1},\rho^2_{\X_2}\}$. Moreover, the resonances are not regularly spaced along $i(-\infty,-\rho_\X]$ and they are generally not among the branching points of the Riemann surface to which the meromorphic extension of the resolvent takes place.

If $z_0$ is a resonance of the Laplacian, then the \textit{(resolvent) residue operator at $z_0$}
is the linear operator 
\begin{equation}
\label{residueop1}
{\Res}_{z_0}: C^\infty_c(\X)\to C^\infty(\X)
\end{equation}
defined by 
\begin{equation}
\label{residueop2}
\big({\Res}_{z_0} f\big)(y)={\Res}_{z=z_0} [R(z)f](y) \qquad (f\in C_c^\infty(\X), \, y\in \X)\,.
\end{equation}
If the meromorphic extension takes place on a Riemann surface, then the right-hand side of 
(\ref{residueop2}) is computed with respect to some coordinate charts and hence determined up to constant multiples. However, the image ${\Res}_{z_0}\big(C^\infty_c(\X)\big)$ is a well-defined
subspace of $C^\infty(\X)$. Its dimension is the \textit{rank} of the residue operator at $z_0$. In the known rank-one and $\SL(3,\R)/\SO(3)$ situations, 
\begin{equation}
\label{residueop3}
{\Res}_{z_0}\big(C^\infty_c(\X)\big)=\{ f \times \varphi_{\lambda_0}; \ f\in  C_c^\infty(\X)\}\,,
\end{equation}
where $\lambda_0\in \mathfrak{a}^*$ depends on $z_0$, $\varphi_{\lambda_0}$ denotes the
spherical function of spectral parameter $\lambda_0$, and $\times$ denote the convolution on $\X$; see \ref{subsection:X} for more information on these objects. For $\X=\X_1\times \X_2$, the expression of the residue operators is generally more complicated. For instance, the convolution is in general taken with a suitable linear combination of spherical functions. See \eqref{eq:Reskeps}.

The group $\G$ acts on the space \eqref{residueop3} by left translation. We prove in Proposition \ref{pro:rank} that this is a generally reducible representation of $\G$ with explicit decomposition as a direct sum of finite-dimensional irreducible spherical representations. It follows, in particular, that all residue operators have finite rank, which might be unexpected for a higher rank symmetric space. Also this representation theoretic aspect of the resonances presents some new features compared to the known situations. Indeed, in the rank-one case with odd multiplicities, the representation spaces arising from the residues are finite-dimensional spherical representations but they are always irreducible; in the case of $\SL(3,\R)/\SO(3)$, they are irreducible and infinite dimensional. 

\smallskip 

This paper is organized as follows. In section \ref{section:preliminaries} we collect some notation and  preliminary results. The case of rank one is reviewed in subsection \ref{subsection:rankone}.
   In section \ref{section:holo-fct}, we write the integral 
formula for resolvent of the Laplacian in polar coordinates on $\mathfrak{a}^*$ and, using \cite{Str05, MV05}, we extend it holomorphically to a logarithmic Riemann surface above $\C \setminus i(-\infty,-L]$ and branched along $(-\infty,0]$. We then prove that, for arbitrarily fixed $f\in C_c^\infty(\X)$ and $y\in \X$, the existence of the meromorphic extension across $i(-\infty,-L]$ of the resolvent $z \to [R(z)f](y)$  is equivalent to that of a certain function $z \to [F(z)f](y)$.
See \eqref{2.1}, Lemma \ref{lemma:holoextF} and Proposition \ref{pro:holoextRF}. Section \ref{section:holo-fct} follows quite closely the construction done for $\SL(3,\R)/\SO(3)$ in \cite[section 2]{HPP14}.  

The following three sections, which form the core of this paper, contain the analysis leading to the extension of the function $z \to [F(z)f](y)$. To simplify our notation, we will omit the dependence of this function on $f$ and $y$. We will therefore indicate it as $z\to F(z)$. This function is defined by an integral over the unit circle $S^1$ in $\mathfrak{a}^*$. Since $\mathfrak{a}^*$ is a 2-dimensional real space, we can identify $\mathfrak{a}^*$ with $\C$ and use complex analysis. The integrand of $F$ is a meromorphic function on $\mathfrak{a}^*$.  By fixing $z$ and deforming $S^1$ into circles of different radius $r>0$, we pick up residues when crossing the singularities of the Plancherel density. Each residue is a function $G_{j,\ell}$ of the variable $z$. Here $\ell\in \Zb_{\geq 0}$ is a nonnegative integral parameter and the index $j\in \{1,2\}$ refers to the fact that the residue comes from a singularity of the Plancherel density of $\X_1$ or $\X_2$. One has therefore to distinguish the case of two odd multiplicities from the simpler case of one odd multiplicity, where one of the two Plancherel density is nonsingular. 

We start with the case of two odd multiplicities, treated first in subsection \ref{subsection:contours-odd-mul} and then completed in section \ref{section:meroext-odd}.  

The number of residue functions $G_{j,\ell}$ one picks up by deforming the circle $S^1$ depends 
a priori on the fixed value $z$ and on the radius $r>0$ of the circle which is the new contour of integration. In the last part of subsection \ref{subsection:contours-odd-mul} we make this dependence local. 

The next step in the analysis is Lemma \ref{lemma:Gjkj}, where we suitably rewrite the functions $G_{j,\ell}$ as meromorphic multi-valued functions. This allow us, in subsection \ref{subsection:meroextGj}, to identify the concrete Riemann surface 
$\M_{j,\ell}$ to which a given function $G_{j,\ell}$ lifts and extends meromorphically. The extended function is denoted by $\wt G_{j,\ell}$.
Its singularities are determined in Lemma \ref{2.8} and the residues at the singularities with respect to a coordinate chart are computed in Lemma \ref{lem:chart-expressions}.

The Riemann surface $\M_{j,\ell}$ is a 2-to-1 cover of $\C\setminus \{0,\pm iL_{j,\ell}\}$, with branching points at $z=\pm iL_{j,\ell}$. The branching points are related to the fact that $G_{j,\ell}$ originates from taking residues at a pole of the Plancherel density of $\X_j$. On the other hand, the singularities of $\wt G_{j,\ell}$ combine the value $L_{j,\ell}$ with the contribution of all the singularities of the Plancherel density of the other symmetric space. It is possible that there are two positive integers $\ell_1$ and $\ell_2$ for which $L_{1,\ell_1}=L_{2,\ell_2}$. This explains why we introduce, in \eqref{eq:Gell} and \eqref{eq:wtGell}, some auxiliary functions $G_{(\ell)}$, $\wt G_{(\ell)}$ and the corresponding concrete Riemann surfaces $\M_{(\ell)}$, to take into account the fact that the $L_{j,\ell}$ might not be all different. 

In subsection \ref{piecewise} we put together the local extensions of the function $F$.
The main result is Proposition \ref{pro:FextendedMainSection-Z}. For each $m\in \Zb_{\geq 0}$ we construct some ``sufficiently thin" neighborhood  of the interval $-i[L_m,L_{m+1})$, where $\{-iL_\ell\}_{\ell \in \Zb_{\geq 0}}$  are  
the different branching points, ordered according to their distance from the origin $0$. Here $L_0=L$. For $z\in W_{(m)}\setminus i\R$, the function $F$ can be written as $F(z)=F_{(m)}(z)+2\sum_{\ell=0}^{m} G_{(\ell)}(z)$, where $F_{(m)}$ is holomorphic. This provides a piecewise extension of $F$ away from $i(-\infty,-L]$. 

The meromorphic extension of $F$ across $i(-\infty,-L]$ is finally obtained in subsection \ref{subsection:extIm<0}. For a fixed positive integer $N$, we construct a Riemann surface 
$\M_{(N)}$ by ``pasting together'' the Riemann surfaces $\M_{(\ell)}$ to which all functions $G_{(l)}$, with $\ell=0,1,\dots,N$, admit meromorphic extension. Moving from branching point to branching point, all the local extensions of $F$ constructed in Proposition \ref{pro:FextendedMainSection-Z} are lifted to a neighborhood $\M_{\gamma_N}$  the branched curve $\gamma_N$ in $\M_{(N)}$ over the interval $-i(0,L_{N+1})$. The different pieces of $F$ have been constructed from intervals of the form $-i[L_\ell,L_{\ell+1})$ with the branching points of the $G_{(\ell)}$'s as endpoints only.  So -- despite the different nature of the branching points, of the meromorphic functions $G_{(\ell)}$ and their singularities --  the structure of the local extensions of $F$ is similar to the one we dealt with in extending the resolvent of the Laplacian of $\SL(3,\R)/\SO(3)$. The same method used to prove \cite[Theorem 19]{HPP14} therefore yields the final formula for the meromorphically extended lift $\wt F$ of $F$, given in Theorem \ref{thm:meroliftF}. The residues of $\wt F$ are computed in Proposition \ref{pro:residueswtF}.

The final subsection \ref{subsection:meroextresolvent} translates the results back to the resolvent. 
See Theorem \ref{thm:meroextshiftedLaplacian}.

In the case of one odd multiplicity, the various steps leading to the extension of the resolvent  can be easily deduced from the corresponding steps in the case of two odd multiplicities we just described.  Supposing that $\X_2=H^{2n+1}(\R)$,  there is one family of residual functions $G_{1,\ell}$. They extend and lift to the same Riemann surfaces $\M_{1,\ell}$ considered above. But there are no contributions from the singularities of the Plancherel density of $\X_2$. Hence the resulting functions $\wt G_{1,\ell}$ (and thus $\wt F$) turn out to be holomorphic. These results are collected in subsection \ref{subsection:contours-even-mul}  and section \ref{section:holoext-even}.

The final section \ref{section:residueops} studies the residue operators and provides their interpretation in terms of representation theory.

\subsubsection*{Acknowledgements} The second author would like to thank the University of Oklahoma for hospitality and financial support. The third author gratefully acknowledges partial support from the NSA grant H98230-13-1-0205. The case of direct products of rank-one Riemannian symmetric spaces of the noncompact type is a step for understanding the general Riemannian symmetric case. In particular, the question about the rank of the residue operators for direct products, was emphasized by A. Strohmaier in his overview talk ``Resonances and symmetric spaces'' at the workshop ``Analysis and Geometry of Resonances'', held at CIRM in March 2015. The authors would like to thank him for his interesting talk and the CIRM for hosting and funding this workshop. 

\medskip

\section{Notation and preliminaries}
\label{section:preliminaries}

\subsection{General notation}
We use the standard notation $\mathbb Z$, $\mathbb Z_{\geq 0}$, $\R$,  $\R^+$, $\C$ and $\C^\times$ for the integers, the nonnegative integers, the reals, the positive reals, the complex
numbers and the non-zero complex numbers, respectively. The upper half-plane in $\C$ is 
$\C^+=\{z \in \C:\Im z>0\}$; the lower half-plane $-\C^+$ is denoted $\C^-$. 
If $\X$ is a manifold, then $C^\infty(\X)$ and $C^\infty_c(\X)$ respectively denote the space of smooth functions and the space of smooth compactly supported functions on $\X$.

\subsection{Analysis on Riemannian symmetric spaces of the noncompact type}
\label{subsection:X}

In this subsection we recall some basic notions on the Riemannian symmetric spaces of the noncompact type and their harmonic analysis. For more information on this subject, we refer the reader to the books \cite{He2, He3, GV88}.

\subsubsection*{Basic structure theory} 
Let $\X$ be a Riemannian symmetric space of the noncompact type. Then $\X=\G/\K$, where $\G$ is a noncompact, connected, semisimple, real Lie group with finite center and $\K$ is a maximal compact subgroup of $\G$. Let $\mathfrak g$ and $\mathfrak k$ ($\subset \mathfrak g$) be the Lie algebras of $\G$ and $\K$, respectively, and let $\mathfrak{g}=\mathfrak{k} \oplus \mathfrak{p}$ be the Cartan decomposition of  $\mathfrak{g}$. 
A Cartan subspace of $\mathfrak{g}$ is a maximal abelian subspace of $\mathfrak{p}$. Fix such 
a subspace $\mathfrak{a}$. Its dimension is called the (real) rank of $\X$.  We denote by $\mathfrak{a}^*$ the (real) dual space of $\mathfrak{a}$ and by $\mathfrak{a}_\C^*$ its complexification. The Killing form of $\mathfrak{g}$ restricts to an inner product on $\mathfrak{a}$. We extend it to $\mathfrak{a}^*$ by duality. The $\C$-bilinear extension of $\inner{\cdot}{\cdot}$ to $\mathfrak{a}_\C^*$ will be indicated by the same symbol. 

The set of (restricted) roots of the pair $(\mathfrak{g},\mathfrak{a})$ is
denoted by $\Sigma$. It consists of all $\beta \in \mathfrak{a}^*$ for
which the vector space $\mathfrak{g}_{\beta} = \{X \in \mathfrak{g}:
\text{$[H,X]=\beta(H)X$  for every $H \in \mathfrak a$}\}$ contains nonzero
elements. 
The dimension $m_{\beta}$ of $\mathfrak{g}_{\beta}$ is called the multiplicity of the root $\beta$. We extend the multiplicities to $\mathfrak{a}^*$ by setting $m_{\beta}=0$ 
if $\beta \in \mathfrak{a}^*$ is not a root. 
Recall that if both $\beta$ and $\beta/2$ are roots, then $m_{\beta/2}$ is even
and $m_{\beta}$ is odd; see e.g. \cite[Ch. X, Ex. F.4., p. 530]{He1}.
We say that a root $\beta$ is unmultipliable 
if $2\beta \notin \Sigma$. 
We fix a set $\Sigma^+$ of positive roots in $\Sigma$. Define
$\mathfrak{a}^*_+=\{\lambda\in \mathfrak{a}^*: \text{$\inner{\lambda}{\beta}>0$ for all $\beta\in\Sigma^+$}\}$.
We respectively denote by $\Sigma_*$ and  $\Sigma^+_*$ the set of unmultipliable roots and unmultipliable positive roots in $\Sigma$. 
Furthermore, we denote by $\rho$ the half-sum of the positive roots, counted with their 
multiplicites: hence
\begin{equation}\label{eq:rho}
\rho=\frac{1}{2}\, \sum_{\beta\in\Sigma_*^+} \big(m_\beta+\frac{m_{\beta/2}}{2}\big)\beta\,.
\end{equation}

The Weyl group $\W$ of the pair $(\mathfrak{g},\mathfrak{a})$ is the finite group of
orthogonal transformations of $\mathfrak{a}$ generated by the reflections in the hyperplanes 
$\ker(\beta)$ with $\beta\in \Sigma$.
The Weyl group action extends to $\frak a^*$ by duality and to $\frak a_\C^*$ by complex linearity.

\subsubsection*{Differential operators}
Let $\mathbb{D}(\X)$ denote the algebra of differential operators on $\X$
which are invariant under the action of $\G$ by left translations.
Then $\mathbb{D}(\X)$ is a commutative algebra which contains the (positive)
Laplacian $\Delta$ of $\X$. Moreover, let $S(\mathfrak{a}_\C)^\W$ be the algebra of 
$\W$-invariant polynomial functions on $\mathfrak{a}_\C^*$. Then there is an isomorphism $\Gamma:\mathbb{D}(\X) \to S(\mathfrak{a}_\C)^\W$ such that $\Gamma(\Delta)(\lambda)=\langle \rho, \rho\rangle- \langle \lambda, \lambda\rangle$ for $\lambda \in \mathfrak{a}^*_\C$. The joint eigenspace $\mathcal E_\lambda(\X)$ for the algebra $\mathbb D(\X)$ is defined by
\begin{equation}\label{eq:El}
\mathcal E_\lambda(\X)=\{f\in C^\infty(X): \text{$Df=\Gamma(D)(\lambda)f$
for all $D \in \mathbb D(\X)$}\}\,. 
\end{equation}
See e.g. \cite[Ch. II, p. 76]{He3}.

The spherical function with the spectral parameter $\lambda\in \mathfrak{a}_\C ^*$ is the unique $K$-invariant function $\varphi_\lambda$ in the joint eigenspace $\mathcal E_\lambda(\X)$ 
satisfying the normalizing condition $\varphi_\lambda(o)=1$, where $o=e\K$ is the base point of $\X$ corresponding to the unit element $e$ of $\G$. It is explicitly given by Harish-Chandra's integral formula. See e.g. \cite[Ch. IV, Proposition 24 and Theorem 4.3]{He2}. 

\subsubsection*{Harish-Chandra's $c$-function}
For $\lambda \in \mathfrak a_\C^*$ and $\beta \in \Sigma$ we shall employ the notation
\begin{equation}\label{eq:lbeta} 
\lambda_\beta=\frac{\inner{\lambda}{\beta}}{\inner{\beta}{\beta}}\,.       
\end{equation}
Let $\beta \in \Sigma_*^+$ and set
\begin{equation}
\label{eq:cbeta} 
c_\beta(\lambda)=\frac{2^{-2\lambda_\beta}\Gamma(2\lambda_\beta)}{\Gamma\big(\lambda_\beta+\frac{m_{\beta/2}}{4}+\frac{1}{2}\big) \Gamma\big(\lambda_\beta+
\frac{1}{2} \big( m_{\beta}+ \frac{m_{\beta/2}}{2}\big)\big)}\,,
\end{equation}
where $\Gamma(t)=\int_0^\infty x^{t-1}e^{-x}\,dx$ is the gamma function.
Harish-Chandra's $c$-function (written in terms of unmultipliable roots) is the function $\cHC$ defined on $\mathfrak{a}_\C^*$ by 
\begin{equation}
\label{eq:c}
\cHC(\lambda)=c^0 \; \prod_{\beta\in\Sigma_*^+} c_\beta(\lambda)\,,
\end{equation}
where $c^0$ is a normalizing constant so that $\cHC(\rho)=1$. 

\subsubsection*{The resolvent of $\Delta$}
Endow
the Euclidean space $\mathfrak{a}^*$ with the Lebesgue measure normalized so that the unit cube has volume $1$. On the Furstenberg boundary $\B=\K/\M$ of $\X$, where $\M$ is the centralizer of $\mathfrak{a}$ in $\K$, we consider the $\K$-invariant measure $db$ normalized so that the volume of $\B$ is equal to $1$.
Let $\X$  be equipped with its (suitably normalized) natural $\G$-invariant Riemannian measure, so that,
by the Plancherel Theorem, the Helgason-Fourier transform
$\mathcal F$ is a unitary equivalence of the Laplacian $\Delta$ on $L^2(\X)$ with the
multiplication operator $M$ on $L^2(\mathfrak{a}^*_+\times B,[\cHC(\bi \lambda)\cHC(-\bi \lambda)]^{-1} \; d\lambda\, db)$ given by
\begin{equation} \label{eq:Mmult}
MF(\lambda,b)=\Gamma(\Delta)(\bi\lambda)F(\lambda,b)=(\inner{\rho}{\rho}+\inner{\lambda}{\lambda})F(\lambda,b)  \qquad ((\lambda,b)\in \mathfrak{a}^*\times B)\,.
\end{equation} 
See \cite[Ch. III, \S 1, no. 2]{He3}.
It follows, in particular, that the spectrum of $\Delta$ is the half-line
$[\rho_\X^2, +\infty[$, where $\rho_\X^2=\inner{\rho}{\rho}$. 
By the Paley-Wiener theorem for $\mathcal{F}$, see e.g. \cite[Ch. III, \S 5]{He3}, for every 
$u \in \C \setminus [\rho_\X^2, +\infty[$ the resolvent of $\Delta$ at $u$ maps $C^\infty_c(\X)$ into $C^\infty(\X)$. 

Recall that for sufficiently regular functions $f_1, f_2:\X \to \C$, the convolution $f_1\times f_2$ is the function on $\X$ defined by $(f_1 \times f_2) \circ \pi= (f_1 \circ \pi) * (f_2 \circ \pi)$. Here $\pi:\G\to \X=\G/\K$ is the natural projection and $*$ denotes the convolution product of functions on $\G$.

The Plancherel formula yields the following explicit expression for the image of $f \in C^\infty_c(\X)$ under the resolvent operator $R(z)=(\Delta-\rho_\X^2-z^2)^{-1}$ of the shifted Laplacian $\Delta - \rho_\X^2$:
\begin{equation}\label{eq:resolvent}
[R(z)f](y)=\int_{\mathfrak{a}^*} \frac{1}{\inner{\lambda}{\lambda}-z^2}\; (f \times \varphi_{\bi\lambda})(y) \; \frac{d\lambda}{\cHC(\bi\lambda)\cHC(-\bi\lambda)}
\qquad (z \in \C^+\,, y\in \X)\,. 
\end{equation}
See \cite[formula (14)]{HP09}. Here and in the following, resolvent equalities as \eqref{eq:resolvent} are given up to non-zero constant multiples.

The convolution $(f\times \varphi_{\bi\lambda})(y)$ can be described in terms of the Helgason-Fourier transform of $f$.
Moreover, for $ f\in C_c^\infty(\X)$, $\lambda \in \mathfrak{a}_\C^*$ and $y=g\cdot o\in \X$, by 
\cite[Ch. III, Lemma 1.2 and proof of Theorem 1.3]{He3}, $f \times \varphi_\lambda$
is the spherical Fourier transform of the $\K$-invariant function $f_y\in C_c^\infty(\X)$
given by
$$f_y(g_1)=\int_\K f(gkg_1 \cdot o) \; dk\,.$$
It follows by the Paley-Wiener Theorem that for every fixed $y\in \X$ the function $(f \times \varphi_{\bi\lambda})(y)$ is a Weyl-group-invariant entire function of $\lambda\in\mathfrak{a}_\C^*$ and there exists a constant $R \geq 0$ (depending on $y$ and on the size of the support of $f$) so that for each
$N \in \mathbb N$
\begin{equation}\label{eq:exptype-conv}
\sup_{\lambda\in\mathfrak{a}_\C^*} e^{-R|{\rm Im} \lambda|}(1+|\lambda|)^N
|(f\times \varphi_{\bi\lambda})(y)| < \infty\,.
\end{equation}

\subsubsection*{Eigenspace representations and convolution operators}
The group $\G$ acts on $\mathcal E_\lambda(\X)$ by left translations:
\begin{equation}\label{eq:eigenrep}
[T_\lambda(g)f](x)=f(g^{-1}x) \qquad (g\in G, \, x\in \X)
\end{equation}
The space $\mathcal E_{\lambda,\G}(\X)$ of $\G$-finite elements
in $\mathcal E_\lambda(\X)$ is a (possibly zero) invariant subspace of
$\mathcal E_\lambda(\X)$. (Recall that $f\in \mathcal E_\lambda(\X)$
is said to be $\G$-finite if the vector space spanned by the left
translates $T_\lambda(g)f$ of $f$ with $g\in \G$ is finite dimensional.)
By definition, $\mathcal E_{w\lambda}(\X)=\mathcal E_{\lambda}(\X)$ and 
$\mathcal E_{w\lambda,\G}(\X)=\mathcal E_{\lambda,\G}(\X)$ for $w\in \W$.

The convolution operator  
\begin{equation}
\label{eq:Rlambda}
\mathcal R_\lambda: \ C^\infty_c(\X) \ni f \to f \times \varphi_\lambda \in C^\infty(\X)
\end{equation}
maps into the eigenspace representation space $\mathcal E_\lambda(\X)$. Its image 
$\mathcal R_\lambda(C^\infty_c(\X))=\{f \times \varphi_\lambda: f \in C^\infty_c(\X)\}$
has been studied in \cite[Theorem 3.2]{HP09} and \cite[Proposition 21]{HPP14}. More precisely, 
$\mathcal R_\lambda(C^\infty_c(\X))$ is a non-zero $T_\lambda$-invariant subspace of $\mathcal E_\lambda(\X)$. Its closure is the unique closed  irreducible subspace $\mathcal E_{(\lambda)}(\X)$ of $\mathcal E_\lambda(\X)$, which is generated by the translates of the spherical function $\varphi_\lambda$. Moreover, the space 
$\mathcal R_\lambda(C^\infty_c(\X))$ is finite dimensional if and only if
$\mathcal E_{\lambda,\G}(\X)\neq \{0\}$ is a finite dimensional spherical representation. 
This means that there is some $w \in \W$ so that $w\lambda-\rho$ is a highest restricted weight. (Recall that an element $\mu \in \mathfrak{a}^*$ is a highest restricted weight if 
$\mu_\alpha \in \Zb_{\geq 0}$ for all $\alpha\in \Sigma^+$.)
In this case, $\mathcal R_\lambda(C^\infty_c(\X))=\mathcal E_{\lambda,\G}(\X)$ is the finite dimensional spherical representation of highest restricted weight $w\lambda-\rho$ for some $w \in \W$. 

\subsection{The rank one case}
\label{subsection:rankone}

The space $\X$ is of rank one when $\a^*$ is one dimensional. In this case, the set $\Sigma^+$ of positive roots consists of at most two elements: $\beta$ and (possibly) $\beta/2$.
The Weyl group $\W$ is $\{\pm \id\}$. 

 Recall the notation $\lambda_\beta=\frac{\langle \lambda,\beta\rangle}{\langle\beta ,\beta\rangle}$.
If we view the real vector space $\a^*$ as a real manifold then the map
\[
\a^*\ni \lambda\to \lambda_\beta\in \R
\]
is a chart with the inverse given by
\begin{equation}\label{1.1}
\R\ni x\to x\beta\in \a^*.
\end{equation}
Also, the Lebesgue measure becomes
\begin{equation}\label{1.2}
d(x\beta)=b\,dx \qquad (x\in \R),
\end{equation}
where
\[
b=\sqrt{\langle\beta ,\beta\rangle}.
\]
Moreover, formula \eqref{eq:rho} becomes $\rho=\rho_\beta \beta$, where
\[
\rho_\beta=\frac{\langle \rho,\beta\rangle}{\langle \beta,\beta\rangle}=\frac{1}{2}\left(m_\beta+\frac{m_{\beta/2}}{2}\right)\,.
\]
Notice that $2\rho_\beta$ is a positive integer. Moreover,
$\rho_{\X}^2=\langle \rho,\rho\rangle=b^2 \rho_\beta^2$ is the bottom of the spectrum of the Laplacian $\Delta$ of $\X$.

\begin{rem} \label{rem:classification}
The following table gives the multiplicity data of the irreducible connected Riemannian symmetric spaces
of rank-one $\G/\K$. In the table, the symbol $\H_0$ denotes the connected component of the identity in a given group $\H$.

\smallskip
\begin{center}
\begin{tabular}{|l|l|c|c|c|l|}
\hline
$\G$ & $\K$ & $\Sigma^+$ & $m_{\beta/2}$ & $m_\beta$ & $\rho_\beta$ \\
\hline
$\SO_0(2n+1,1)$, $n \geq 1$ & $\SO(2n+1)$ & $\{\beta\}$ & $0$ & $2n$ & $n$ \\
\hline
$\SO_0(2n,1)$, $n \geq 1$ & $\SO(2n)$ & $\{\beta\}$ & $0$ & $2n-1$ & $n\,-\,1/2$ \\
\hline
$\SU(n,1)$, $n \geq 2$ & $\Sg(\Ug(n)\times \Ug(1))$ & $\{\beta/2,\beta\}$ & $2(n-1)$ & 1 & $n/2$\\
\hline
$\Sp(n,1)$, $n \geq 2$ & $\Sp(n)\times \Sp(1)$ & $\{\beta/2,\beta\}$ & $4(n-1)$ & 3 & $n\,+\,1/2$ \\
\hline
$\Fg_{4(-20)}$ & $\Spin(9)$       & $\{\beta/2,\beta\}$ & $8$ & $7$ & $11/2$\\
\hline
\end{tabular}
\end{center}
\end{rem}
\bigskip

In the rank-one case, Harish-Chandra's $c$ function reduces to a constant multiple of the function $c_\beta$, see \eqref{eq:cbeta}. The Plancherel density can be rewritten as a polynomial multiple of a hyperbolic function by means of the classical formulas for the Gamma function. Indeed, if $m_\beta$ is even define
\[
P(x)=\prod_{k=0}^{2(\rho_\beta-1)}\big(x-(\rho_\beta-1)+k\big) \qquad
\text{and} \qquad Q(x)=1.
\]
If $m_\beta$ is odd, then $\frac{m_\beta+1}{2}\leq 2\rho_\beta$ and we define
\begin{eqnarray*}
P(x) &=&\prod_{k=0}^{2\rho_\beta-2}\big(x-(\rho_\beta-1)+k\big)
\prod_{k=0}^{\frac{m_{\beta/2}}{2}-1} \big(x-(\frac{m_{\beta/2}}{4}-\frac{1}{2}\big)+k\big)\\
&=&\prod_{k=\frac{m_\beta+1}{2}}^{2\rho_\beta-\frac{m_\beta+1}{2}}(x-\rho_\beta+k)\prod_{k=1}^{2\rho_\beta-1}(x-\rho_\beta+k)
\end{eqnarray*}
and
\[
Q(x)=\cot(\pi(x-\rho_\beta)).
\]
Notice that if $Q\ne 1$ then the product  $PQ$ is singular at
\[
i(\rho_\beta+\Zb) \setminus\{-i(\rho_\beta-1), -i(\rho_\beta-2),\dots, i(\rho_\beta-2), i(\rho_\beta-1)\}.
\]
The Plancherel density becomes
\begin{equation}
\label{eq:Plancherel-density}
\frac{1}{\cHC(\bi\lambda)\cHC(-\bi\lambda)}=c_0\,\lambda_\beta\,P(i\lambda_\beta)\,Q(i\lambda_\beta) \qquad (\lambda\in\a_\C^*).
\end{equation}
where $c_0$ is a constant depending on the root multiplicities.

To simplify the notation in the later parts of the paper, we introduce the functions 
\begin{equation}
\label{eq:pq}
 p(x)=P\big(i\tfrac{x}{b}\big)\qquad \text{and} \qquad q(x)=Q\big(i\tfrac{x}{b}\big).
\end{equation}

\begin{rem}
\label{rem:parity-pq}
The function $q$ is odd if $m_\beta$ is odd because $2\rho_\beta \in \Zb$ and the cotangent is odd and $\pi$-periodic. It is obviously even when $m_\beta$ is even. Moreover, $xp(x)q(x)$ is even since, by (\ref{eq:Plancherel-density}), it is a constant multiple of $[\cHC(\bi\lambda)\cHC(-\bi\lambda)]^{-1}$ where $\lambda_x=\frac{x}{b}\,\beta$.
It follows that the polynomial $p$ is odd if $m_\beta$ is even, and even if $m_\beta$ is odd.
Of course this can also be seen directly. Indeed, suppose $m_\beta$ is odd. Then the first factor of $P$ is odd if and only if
$2\rho_\beta-m_\beta-1\in \Bbb Z_{\geq 0}$, i.e.  $\frac{m_{\beta/2}}{2} \in 2\Zb+1$, whereas the second factor is odd if and only if $\rho_\beta \in \Zb$, i.e. $m_\beta+\frac{m_{\beta}}{2} \in 2\Zb$. Thus the two factors are either both odd (and this happens if $\frac{m_{\beta/2}}{2} \in 2\Zb+1$) or both even (if $\frac{m_{\beta/2}}{2} \in 2\Zb$). In any case, $P$ is even when $m_\beta$ is odd. When $m_\beta$ is even, then $\rho_\beta \in \Zb$, so $P$ is odd.
\end{rem}

The following theorem was proven in \cite{HP09}, see also \cite{MW00}. Here we employ the notation $\beta/2$, $\beta$ instead of $\alpha$, $2\alpha$ for the elements of $\Sigma^+$. This choice allows us to unify the resulting formulas. (In particular, there is no need of distinguishing, as in \cite{HP09}, between the cases $m_{2\alpha}=0$ and $m_{2\alpha}\neq 0$ when $m_\alpha$ is odd.) When $m_\beta$ is even (and
hence $\beta/2$ is not a root), we are in a case of even multiplicities and odd rank, in which $R(z)$ admits holomorphic extension to the entire complex plane. We can therefore restrict ourselves to the case where $m_\beta$ is odd. 

\begin{thm}
Suppose $\X$ has rank one and $m_\beta$ is odd. Then the resolvent $R(z)$ admits meromorphic extension from $\C^+$ to the entire complex plane $\C$, with simple poles at the points 
\begin{equation}
\label{eq:poles}
z_k=-i(\rho_\beta+k)b
 \qquad (k\in \Zb_{\geq 0})\,.
\end{equation}
The (resolvent) residue operator
$R_k: C^\infty_c(\X) \to C^\infty(\X)$
given by 
\[
[R_k(f)](y)={\Res}_{z=z_k} [R(z)f](y)
\]
has image 
\[
R_k\big(C^\infty_c(\X)\big)=\{ f\times \varphi_{(\rho_\beta+k)\beta}: f\in C^\infty_c(\X)\}\,,
\]
where $\varphi_{(\rho_\beta+k)\beta}$ is the spherical function of spherical parameter 
$(\rho_\beta+k)\beta$. Endowed with the action of group $\G$ by left translations, the image $R_k\big(C^\infty_c(\X)\big)$ is the finite dimensional spherical representation of $\G$ of highest restricted weight $k\beta$. 
\end{thm}
\begin{proof}
This is \cite[Theorem 3.8]{HP09}. Notice that for $k\in \Zb_{\geq 0}$, the element $(\rho_\beta+k)\beta-\rho$ is a highest restricted weight as 
$\big((\rho_\beta+k)\beta-\rho\big)_\beta=(\rho_\beta+k)\beta_\beta-\rho_\beta=k\in \Zb_{\geq 0}$.
(This implies the integrality condition for $\beta/2$ as well, since $\mu_{\beta/2}=2\mu_\beta$ for $\mu\in \mathfrak{a}^*$.) 
\end{proof}


\subsection{The functions $\cz$ and $\sz$}
For $r>0$ and $a,b \in \R\setminus \{0\}$ set
\begin{eqnarray*}
\Dg_r&=&\{z\in \C;\ |z|<r\}\,,\\
\Eg_{a,b}&=&\big\{\xi+i\eta\in \C;\ \left(\tfrac{\xi}{a}\right)^2+ \left(\tfrac{\eta}{b}\right)^2<1\big\}\,.
\end{eqnarray*}
Their boundaries $\partial \Dg_r$ and $\partial \Eg_{a,b}$ are respectively the circle of radius $r$ and the ellipse of semi-axes $|a|,|b|$, both centered at $0$. In particular, $\partial \Dg_1$ is the unit circle $\Sg^1=\{z\in \C:|z|=1\}$.
Moreover, the closure $\overline{\Dg_r}$ of $\Dg_r$ is the closed disc of center $0$ and radius $r$.

For $z \in \C^\times$ define
\begin{equation}\label{def-of-c-and-s}
\cz(z)=\frac{z+z^{-1}}{2}\quad \text{and} \quad
\sz(z)=\frac{z-z^{-1}}{2}=i\cz(-iz)\,.
\end{equation}
Then $\cz:\Dg_1\setminus \{0\} \to \C\setminus [-1,1]$ is a biholomorphic map.
For $0<r<1$, it restricts to a biholomorphic function
\begin{equation}\label{first bijection}
\cz:\Dg_1\setminus\overline{\Dg_r}\to  \Eg_{\cz(r), \sz(r)}\setminus [-1,1]
\end{equation}
and to a bijection
\begin{equation}\label{second bijection}
\cz:\partial \Dg_r\to \partial  \Eg_{\cz(r), \sz(r)}.
\end{equation}
Let $\cz^{-1}:\C\setminus [-1,1] \to \Dg_1\setminus \{0\}$ be the inverse of the function $c$.
We will need the following lemma, proved in \cite[Lemma 7]{HPP14}.
\begin{lem}\label{rays and circles}
Let $|\zeta|=|\zeta_0|=1$. Then
\[
\left(\zeta \cz^{-1}(\zeta_0 \R\setminus[-1,1])\right)\cap \left(\cz^{-1}(\zeta_0 \R\setminus[-1,1])\right)\ne\emptyset
\]
if and only if $\zeta=\pm 1$.
\end{lem}

Let $\sqrt{\cdot}$ denote the single-valued holomorphic branch of the square root function defined on $\C \setminus (-\infty,0]$ by
\begin{equation}\label{square root 1}
\sqrt{Re^{i\Theta}}=\sqrt{R}\,e^{\frac{i\Theta}{2}} \qquad (R>0, \,-\pi<\Theta<\pi).
\end{equation}
Then the function $\sqrt{z+1}\sqrt{z-1}$, originally defined on $\C\setminus (-\infty, 1]$, extends to a holomorphic function on $\C\setminus [-1,1]$ satisfying
\begin{equation}
\label{eq:sign-sqrt-sqrt}
\sqrt{(-z)+1}\sqrt{(-z)-1}=-\sqrt{z+1}\sqrt{z-1}\,.
\end{equation}
See e.g. \cite[Lemma 5]{HPP14}.
For $z\in\C\setminus [-1,1]$, we have
by \cite[Lemma 6]{HPP14}:
\begin{eqnarray}
\label{eq:cinverse}
&&\cz^{-1}(z)=z-\sqrt{z+1}\sqrt{z-1},\\
\label{eq:scinverse}
&&\sz\circ \cz^{-1}(z)=-\sqrt{z+1}\sqrt{z-1}.
\end{eqnarray}
Consider the Riemann surface
\begin{equation}\label{preimageofzero}
\M=\{(w,\zeta)\in \C^2, \zeta^2=w^2-1\}
\end{equation}
above $\C$, with holomorphic projection map
\begin{equation}\label{coveringMtoC}
\pi: \M\ni (w,\zeta)\to w\in \C\,.
\end{equation}
The fibers of $\pi$ consist of two points $(w,\zeta)$ and $(w,-\zeta)$ if $w\ne \pm 1$. If $w=\pm 1$, then the fibers consist of one point $(w,0)$.

Let $\Sg\subseteq \C$ and let $\wt \Sg\subseteq \M$ be the preimage of $\Sg$ in $\M$ under the map $\pi$ given by (\ref{coveringMtoC}). We say that a function $\wt f: \wt \Sg\to \C$ is a lift of $f:\Sg\to \C$ if there is a holomorphic section $\sigma: \Sg \to \wt \Sg$ of the restriction of
$\pi$ to $\wt \Sg$ so that
\[
\wt f(w,\zeta)=f(w) \qquad ((w,\zeta)\in \sigma(\Sg))\,.
\]
Let $\sigma^+:\C\setminus [-1,1] \ni w \to (w,\zeta^+(w)) \in \M$ be the holomorphic section of $\pi$ defined by $\zeta^+(w)=\sqrt{w+1}\sqrt{w-1}$.
Because of \eqref{eq:cinverse} and \eqref{eq:scinverse}, the functions
\begin{eqnarray}
\label{eq:lift of c}
&&(\cz^{-1})\,\wt{}:\M\ni (w,\zeta)\to w-\zeta\in \C\,,\\
\label{eq:lift of sc}
&&(\sz\circ \cz^{-1})\,\wt{}:\M\ni (w,\zeta)\to -\zeta\in \C
\end{eqnarray}
are holomorphic extensions to $\M$ of the lifts of $\cz^{-1}$ and
$\sz\circ \cz^{-1}$ for $\sigma^+$, respectively.


\section{Direct products: extension away from the negative imaginary axis }
\label{section:holo-fct}

From now on 
$\X= \X_1\times \X_2$ where $\X_1$ and $\X_2$ are Riemannian symmetric spaces of the noncompact type and of rank one. 
To distinguish the objects associated with the two spaces, we will add the indices $\null_1$ and $\null_2$ to the notation introduced above. (However, we will write $\a_{1,\C}$ rather than ${\a_1}_\C$ or $(\a_1)_\C$.)  Hence, we have
\begin{eqnarray*}
&&\a^*=\a^*_1\oplus \a^*_2,\ \ \ \langle\cdot,\cdot\rangle=\langle\cdot,\cdot\rangle_1\oplus\langle\cdot,\cdot\rangle_2,\ \ \  \Delta=\Delta_1\otimes \id + \id\otimes \Delta_2,\\
&&\cHC(\lambda)=c_{\rm{HC},1}(\lambda_1)c_{\rm{HC},2}(\lambda_2)\qquad (\lambda_1\in \a_{1,\C}^*,\, \lambda_2\in \a_{2,\C}^*, \,\lambda=\lambda_1+\lambda_2)\,,\\
&&\rho=\rho_1+\rho_2=\rho_{\beta_1}\beta_1+\rho_{\beta_2}\beta_2,\ \ \ b_1=\sqrt{\langle \beta_1,\beta_1\rangle},\ \ \ b_2=\sqrt{\langle \beta_2,\beta_2\rangle}
\end{eqnarray*}
and the Harish-Chandra spherical function
\begin{multline*}
\varphi_\lambda(y)=\varphi_{1,\lambda_1}(y_1)\varphi_{2,\lambda_2}(y_2)\\
(\lambda_1\in \a_{1,\C}^*, \,\lambda_2\in \a_{2,\C}^*, \, \lambda=\lambda_1+\lambda_2,\, y_1\in\X_1, \, y_2\in\X_2, \,y=(y_1,y_2))
\end{multline*}
Then $\rho_\X^2=\langle \rho,\rho\rangle=\langle \rho_1,\rho_1\rangle+ \langle \rho_2,\rho_2\rangle= b_1^2\rho_{\beta_1}^2+b_2^2\rho_{\beta_2}^2$ is the bottom of the spectrum of $\Delta$.

As in the case of $\mathrm{SL}(3,\R)/\mathrm{SO}(3)$, treated in \cite{HPP14}, it will be convenient to identify $\mathfrak{a}^*$ with $\C$ as vector spaces over $\R$. More precisely, 
we want to view $\mathfrak{a}_1^*$ and $\mathfrak{a}_2^*$ as the real and the purely imaginary axes, respectively. To distinguish the resulting complex structure in $\mathfrak{a}^*$ from the natural complex structure of $\mathfrak{a}^*_\C$,
we shall indicate the complex units in $\mathfrak{a}^*\equiv \C$ and $\mathfrak{a}_\C^*$ by $i$ and $\bi$, respectively.
So $\mathfrak{a}^*\equiv \C=\R +i \R$, whereas $\mathfrak{a}_\C^*=\mathfrak{a}^*+\bi\mathfrak{a}^*$.
For $r,s\in \R$ and $\lambda,\nu\in \mathfrak{a}^*$ we have $(r+is)(\lambda+\bi\nu)=(r\lambda-s\nu)+\bi(r\nu+s\lambda)\in \mathfrak{a}^*_\C$.


\subsection{The resolvent kernel}
\label{subsection:resolvent kernel}

Introduce the coordinates \eqref{1.1} on each component of the real vector space $\a^*=\a_1^*\oplus \a_2^*$. Let $f\in C_c^\infty(\X)$, $y\in \X$ and $z\in \C^+$.
Using \eqref{eq:Plancherel-density} and omitting non-zero constant multiples, we can rewrite \eqref{eq:resolvent} as 
\begin{multline*}
R(z)f(y)=
\int_{\a^*}\frac{1}{\langle \lambda,\lambda\rangle-z^2}(f\times \varphi_{\bi\lambda})(y)\frac{1}{\cHC(\bi\lambda)\cHC(-\bi\lambda)}\,d\lambda\\
=\int_{\R^2}\frac{1}{x_1^2b_1^2+x_2^2b_2^2-z^2}(f\times \varphi_{\bi x_1\beta_1+\bi x_2\beta_2})(y)x_1x_2P_1(ix_1)P_2(ix_2)Q_1(ix_1)Q_2(ix_2)\,dx_1\,dx_2\,.
\end{multline*}

The $\langle\cdot ,\cdot \rangle$-spherical coordinates on $\a^*$ become elliptical coordinates on $\R^2$ via the substitution
\[
x_1=\frac{r}{b_1}\cos \theta,\ \ \ x_2=\frac{r}{b_2}\sin \theta \qquad (0<r\,,\ 0\leq \theta<2\pi).
\]
In these terms (up to a non-zero constant multiple)
\[
R(z)f(y)=\int_0^\infty\frac{1}{r^2-z^2}F(r)\,r\,dr,
\]
where
\begin{eqnarray}\label{2.1}
F(r)&=&\int_0^{2\pi}(f\times \varphi_{\bi\frac{r}{b_1}\cos \theta\,\beta_1+\bi\frac{r}{b_2}\sin \theta\,\beta_2})(y)\,r^2\cos\theta \sin\theta\\
&&\times \; p_1(r\cos \theta) q_1(r\cos \theta) p_2(r\sin \theta) q_2(r\sin \theta)\,d\theta\,.\nn
\end{eqnarray}
Here and in the following, we omit from the notation the dependence of $F$ on the function $f\in C^\infty_c(\X)$ and on $y\in \X$.

Recall the functions, \eqref{def-of-c-and-s},
\[
\cz(w)=\frac{w+w^{-1}}{2}\,\qquad  \sz(w)=\frac{w-w^{-1}}{2}=i\cz(-iw) \qquad (w\in\C^\times)
\]
and notice that
\[
\cos\theta=\cz(e^{i\theta})\,, \qquad \sin\theta=\frac{\sz(e^{i\theta})}{i}=\cz(-ie^{i\theta})\,,\qquad d\theta=\frac{d e^{i\theta}}{ie^{i\theta}}.
\]
For $z \in \C$ and $w \in \C^\times$ define
\begin{eqnarray}
\label{eq:psiz}
\psi_z(w)&=&(f\times \varphi_{\bi\frac{z}{b_1}\cz(w)\,\beta_1+\bi\frac{z}{b_2}\cz(-iw)\,\beta_2})(y)\\
\label{eq:phiz}
\phi_z(w)&=&-z^2\cz(w) \frac{\sz(w)}{w} p_1\big(z\cz(w)\big) q_1\big(z\cz(w)\big) p_2\big(z\cz(-iw)\big) q_2\big(z \cz(-iw)\big).
\end{eqnarray}
Then
\begin{equation}\label{F(r)}
F(r)=\int_{|w|=1}\psi_r(w)\phi_r(w)\,dw\,.
\end{equation}
Notice that  $\cz(-w)=-\cz(w)$. Hence, by Remark \ref{rem:parity-pq} and the Weyl group invariance of the spherical functions with respect to the spectral parameter, we have
\begin{eqnarray}
\label{eq:parity-psi}
\psi_{-z}(w)=\psi_{z}(w) \qquad &&\psi_{z}(-w)=\psi_{z}(w)\,,\\
\label{eq:parity-phi}
\phi_{-z}(w)= \phi_{z}(w) \qquad &&\phi_{z}(-w)= -\phi_{z}(w)\,.
\end{eqnarray}

For $j=1,2$ let
\begin{equation}\label{eq:Lj}
L_{j,0}=\begin{cases}
b_j\rho_{\beta_j}=\sqrt{\langle \rho_j,\rho_j\rangle} &\text{if $m_{\beta_j}\in 2\Bbb Z+1$},\\
+\infty  &\text{if $m_{\beta_j}\in 2\Bbb Z$}
\end{cases}
\end{equation}
and let
\begin{equation}
\label{eq:Sj}
S_j=
\begin{cases}
ib_j((\rho_{\beta_j}+\Bbb Z_{\geq 0})\cup(-\rho_{\beta_j}-\Bbb Z_{\geq 0})) &\text{if $m_{\beta_j}\in 2\Bbb Z+1$},\\
\emptyset &\text{if $m_{\beta_j}\in 2\Bbb Z$}.
\end{cases}
\end{equation}
(The index $0$ in \eqref{eq:Lj} will play a role later in this paper, where $L_{j,0}$ will be the first element of an infinite series $L_{j,\ell}$. See \eqref{eq:Ljl}.)

\subsection{Holomorphic extension}
\label{section:holo-ext}
We start our extension procedure with a two step holomorphic extension of $R(z)$ to a logarithmic Riemann surface branched along $(-\infty,0]$.

\begin{lem}
\label{lemma:holoextF}
The function $F(r)$, \eqref{F(r)}, extends holomorphically to
\begin{eqnarray}\label{2.2}
F(z)&=&\int_{|w|=1}\psi_z(w)\phi_z(w)\,dw\,,
\end{eqnarray}
where
\[
z\in \C\setminus i((-\infty,- L]\cup[L, +\infty)),\ \ \ L=\min\{L_{1,0}, L_{2,0}\}.
\]
The function $F(z)$ is even and $F(z)z^{-2}$ is bounded near $z=0$.
\end{lem}
\begin{prf}
The holomorphic extension of $F$ to $\C\setminus i((-\infty,- L]\cup[L, +\infty))$ is
a consequence of the fact that $p_j q_j$ has singularities only at points of $S_j$.
Since $\psi_z(w)$ and $\phi_z(w)$ are even functions of $z$, so is $F(z)$. Finally,
\eqref{eq:phiz} implies that $F(z)z^2$ is bounded near $z=0$.
\end{prf}

Let $f\in C^\infty_c(\X)$ and $y\in \X$ be fixed.
The following proposition, which is the analogue to \cite[Proposition 2, (a)]{HPP14}, reduces the study of the meromorphic extension from $\C^+$ to $\C \setminus (-\infty,0]$ of the resolvent
$z \to [R(z)f](y)$ to that of the function $z\to F(z)$.
Recall that we are omitting the dependence on $f$ and $y$ from the notation. We will do this
for the resolvent as well, writing $R(z)$ instead of $[R(z)f](y)$.

\begin{pro}
\label{pro:holoextRF}
Let $f\in C^\infty_c(\X)$ and $y\in \X$ be fixed.
Fix $x_0>0$ and $y_0>0$. Let
\begin{eqnarray*}
Q&=&\{z\in \C; \Re z>x_0,\ y_0> \Im z \geq 0\}\\
U&=&Q\cup  \{z\in \C; \Im z <0\}
\end{eqnarray*}
Then there is a holomorphic function $H:U\to \C$ such that
\begin{equation}
\label{eq:holoextRF}
R(z)=H(z)+\pi i\, F(z) \qquad (z\in Q).
\end{equation}
As a consequence, the resolvent $R(z)=[R(z)f](y)$ extends holomorphically from $\C^+$ to
$\C \setminus \big((-\infty,0] \cup i(-\infty, -L]\big)$.
\end{pro}
\begin{prf}
Since
\[
\frac{2r}{r^2-z^2}=\frac{1}{r-z}+\frac{1}{r+z},
\]
we have
\begin{equation} \label{original integral}
2 R(z)=\int_0^\infty \frac{F(r)}{r-z}\,dr+\int_0^\infty \frac{F(r)}{r+z}\,dr\,,
\end{equation}
where the first integral is holomorphic in $\C\setminus [0,+\infty)$ and the
second is holomorphic in $\C\setminus (-\infty,0]$.

Let $\gamma_+$ be a curve in the first quadrant that starts at $0$, goes to the right and up above $z_0=x_0+iy_0$, and then becomes parallel to the positive real line and goes to infinity. We suppose that $Q$ is in the interior to the region bounded by $\gamma_+$ and the positive real axis.

Let $M, m$ be two fixed positive numbers. The convolution defining the function $\psi_z(w)$
in \eqref{eq:psiz} can be written in terms of the Helgason-Fourier transform of the function
$f\in C_c^\infty(\X)$; see \cite[Ch. III, Lemma 1.2 and proof of Theorem 1.3]{He3}. Therefore,  by the Paley-Wiener theorem for this transform (see \cite[Theorem 5.1, p. 260]{He3}), $\psi_z(w)$  is rapidly decreasing in the strip $\{z\in \C; |\Im z|\leq M\}$. See also \cite[formula (22)]{HPP14}.
Moreover, $\phi_z(w)$ is a polynomial function of $z$ times $q_1\big(z\cz(w)\big) q_2\big(z\cz(-iw)\big)$. This latter function is bounded in the half plane  $\{z \in \C; \Re z \geq m\}$. These estimates allow us to apply Cauchy's theorem:
for $z\in Q$, we have
\begin{equation}\label{first integral computed}
\int_0^\infty \frac{F(r)}{r-z}\,dr=\int_{\gamma_+} \frac{F(\zeta)}{\zeta-z}\,d\zeta + 2\pi i F(z)\,,
\end{equation}
where the first integral extends holomorphically from the interior of $Q$ to $U$.
The proposition then follows, with
\[
H(z)=\frac{1}{2}\left(\int_0^\infty \frac{F(r)}{r+z}\,dr+\int_{\gamma_+} \frac{F(\zeta)}{\zeta-z}\,d\zeta\right)\,.
\]
\end{prf}

As in the case of $\SL(3,\R)/\SO(3)$ in \cite{HPP14}, the extension of $R(z)$ across $(-\infty,0]$ can be deduced from the results of Mazzeo and Vasy \cite{MV05} and of Strohmaier \cite{Str05}. Another option is to rewrite Proposition \ref{pro:holoextRF} with the region $Q$ replaced by $\{z\in\C;\ \Re z<-x_0,\ y_0> \Im z \geq 0\}$.

Let $\log$ denote the holomorphic branch of the logarithm defined on
$\C \setminus ]-\infty, 0]$ by $\log 1=0$.
It gives a biholomorphism between $\C^+$ and the strip $S_{0,\pi}=\{\tau \in \C: 0<\Im \tau < \pi\}$. As above, $f\in C^\infty_c(\X)$ and $y\in \X$ are fixed and omitted from the notation. Set $\tau=\log z$ and define
\begin{equation}
\label{eq:R-log}
\Rlog\null(\tau)=R(e^{\tau})=
\int_{\mathfrak{a}^*} \frac{1}{\inner{\lambda}{\lambda}-e^{2\tau}}\; (f \times \varphi_{\bi\lambda})(y) \; \frac{d\lambda}{\cHC(\bi\lambda)\cHC(-\bi\lambda)}\,.
\end{equation}
Polar coordinates in $\mathfrak{a}^*$ now give
\begin{equation}
\label{eq:R-log-polar}
\Rlog(\tau)=\int_{-\infty}^{+\infty} \frac{1}{e^{2t}-e^{2\tau}}\; F(e^t) e^{2t} \; dt\,.
\end{equation}
Since $F$ is even, the function $t\to F(e^t)$ is $i\pi$-periodic.

\begin{pro}
\label{pro:holoextRFnegaxis}
Let $f\in C^\infty_c(\X)$ and $y\in \X$ be fixed.
The function $\Rlog(\tau)=[\Rlog(\tau)f](y)$ extends
holomorphically from $S_{0,\pi}$ to the open set
\[
\Ulog=\C \setminus \bigcup_{n \in \mathbb Z\setminus \{0\}}
\Big(i\pi \big(n+\tfrac{1}{2}\big)+[\log(L),+\infty)\Big)
\]
and satisfies the identity:
\begin{equation}
\Rlog(\tau+i\pi)=\Rlog(\tau)+i\pi F(e^\tau)\qquad (\tau \in \Ulog\setminus \big(i\pi \big(\tfrac{1}{2}\big)+[\log(L),+\infty\big))\,.
\end{equation}
Consequently, the resolvent $R(z)=[R(z)f](y)$ extends holomorphically from $\C \setminus \big((-\infty,0] \cup i(-\infty, -L]\big)$ to a logarithmic Riemann
surface branched along $(-\infty,0]$, with the preimages of  $i\big((-\infty, -L]\cup [L, +\infty)\big)$ removed and, in terms of monodromy, it satisfies the following equation
\begin{equation*}
R(z e^{2i\pi})=R(z)+2 i\pi\, F(z)\\ \qquad (z \in \C \setminus  \big((-\infty,0]\cup i(-\infty, -L]\cup i[L,+\infty) \big)).
\end{equation*}

\end{pro}
\begin{prf}
This is \cite[Proposition 4.3]{Str05} with $f(x)=F(x)x$ for $x \in [0,+\infty)$; see also \cite[Theorem 1.3]{MV05}. Recall that $F(0)=0$.
\end{prf}

Proposition \ref{pro:holoextRF} shows that all possible resonances of the resolvent $R$ of the Laplacian of $\X$ are located along the half-line $i(-\infty, -L]$. Because of (\ref{eq:holoextRF}),
the possible meromorphic extension of $R$ across this domain is equivalent to that of the
function $F$. In fact, we shall see that the function $F$, and then the resolvent, extend further along the negative imaginary axis to a Riemann surface above $\C$. The extension is holomorphic 
except when both multiplicities $m_{\beta_1}$ and $m_{\beta_2}$ are odd. In the latter case, we shall prove that the extension is meromorphic, with simple poles, and first pole (i.e. the first resonance) occurs at $-i(L_{1,0}+L_{2,0})$. Observe that 
$
L_{1,0}+L_{2,0}>L=\min\{L_{1,0},L_{2,0}\}
$.

If both $\X_1$ and $\X_2$ have even multiplicities, then Proposition \ref{pro:holoextRFnegaxis}
completely describes the holomorphic extension of the resolvent to a logarithmic Riemann surface 
above $\C$; see \cite[Theorem 3.3(2)]{Str05}. Therefore, in the following, we shall assume that at least one of the numbers $m_{\beta_1}$ or $m_{\beta_2}$ is odd. 

\section{Contour deformation and residues}
\label{section:contour}

The following proposition is a consequence of the Residue Theorem.

\begin{pro}\label{2.5}
Suppose $z\in \C\setminus i((-\infty,- L]\cup[L, \infty))$ and $r>0$ are such that
\begin{equation}\label{2.5.1}
(S_1\cup S_2)\cap z\partial E_{\cz(r),\sz(r)}=\emptyset.
\end{equation}
Then
\begin{equation}
\label{eq:F-contour}
F(z)=F_r(z)+2\pi i \, G_r(z),
\end{equation}
where
\begin{eqnarray*}
F_r(z)&=&\int_{\partial D_r}\psi_z(w)\phi_z(w)\,dw,\\
G_r(z)&=&{\sum}_{w_0}'\psi_z(w_0)\Res_{w=w_0}\phi_z(w),
\end{eqnarray*}
and $\sum_{w_0}'$ denotes the sum over all the $w_0$ such that
\begin{equation}\label{2.5.2}
z\cz(w_0)\in S_1\cap z(\Eg_{\cz(r),\sz(r)}\setminus [-1,1])
\end{equation}
or
\begin{equation}\label{2.5.3}
z\cz(-iw_0)\in S_2\cap z(\Eg_{\cz(r),\sz(r)}\setminus [-1,1]).
\end{equation}
Both $F_r$ and $G_r$ are holomorphic functions on the open subset of $\C\setminus i((-\infty,- L]\cup[L, \infty))$ where the condition \eqref{2.5.1} holds. Furthermore, $F_r$ extends to a holomorphic function on the open subset of $\C$ where the condition \eqref{2.5.1} holds.
\end{pro}
Let $z\in \C^\times$. By Lemma \ref{rays and circles} (with $\zeta_0=i \frac{|z|}{z}$ and $\zeta=i$), there is no $w_0$ satisfying both $z\cz(w_0) \in i\R$ and
$z\cz(-iw_0) \in i\R$. In particular, there is no $w_0$ which satisfies both \eqref{2.5.2} and \eqref{2.5.3}.
This implies that the set of singularities of the functions $w\to
q_1\big(z\cz(w)\big)$ and $w\to q_2\big(z\cz(-iw)\big)$ are disjoint. 
 Hence, we deduce the following lemma.
\begin{lem}\label{2.6}
With the notation of Proposition \ref{2.5} we have
\begin{eqnarray*}
{\sum}_{w_0}'\psi_z(w_0)\Res_{w=w_0}\phi_z(w)
={\sum}_{w_1}'\psi_z(w_1)\Res_{w=w_1}\phi_z(w)
+{\sum}_{w_2}'\psi_z(w_2)\Res_{w=w_2}\phi_z(w),
\end{eqnarray*}
where
\begin{equation}\label{2.6.1}
z\cz(w_1)\in S_1\cap z(\Eg_{\cz(r),\sz(r)}\setminus [-1,1])
\end{equation}
and
\begin{equation}\label{2.6.2}
z\cz(-iw_2)\in S_2\cap z(\Eg_{\cz(r),\sz(r)}\setminus [-1,1]).
\end{equation}
\end{lem}

The remainder of this section is devoted to the explicit computation of the function $G_r$
occurring in \eqref{eq:F-contour}. We have to distinguish the case when both multiplicities
$m_{\beta_1}$ and $m_{\beta_2}$ are odd, from the case when one of them is even.

\subsection{The case of two odd multiplicities}
\label{subsection:contours-odd-mul}

In this subsection we suppose that both multiplicities
$m_{\beta_1}$ and $m_{\beta_2}$ are odd.

\begin{lem}\label{2.7}
With the notation of Proposition \ref{2.5}, \eqref{2.6.1} and \eqref{2.6.2}, we have for $z\in \C\setminus i((-\infty,- L]\cup[L, +\infty))$,
\begin{equation}\label{2.7.1}
\Res_{w=w_1}\phi_z(w)=\frac{ib_1}{\pi}z\cz(w_1) p_1\big(z\cz(w_1)\big) p_2\big(z\cz(-iw_1)\big) q_2\big(z\cz(-iw_1)\big)
\end{equation}
and
\begin{equation}\label{2.7.2}
\Res_{w=w_2}\phi_z(w)=\frac{b_2}{i\pi}z\cz(-iw_2) p_1\big(z\cz(w_2)\big) p_2\big(z\cz(-iw_2)\big) q_1\big(z\cz(w_2)\big)
\end{equation}
\end{lem}
\begin{prf}
Let $w_1$ be satisfying  condition \eqref{2.6.1}. As observed before the proof of Lemma \ref{2.6}, the function $q_2$ is not singular at $z\cz(-iw_1)$. Furthermore,
\begin{eqnarray*}
\Res_{w=w_1} q_1\big(z\cz(w)\big)=\Res_{w=w_1} \frac{1}{i} \, \coth\big(\pi\big(\frac{z}{b_1}\cz(w)+i\rho_{\beta_1}\big)\big)
=\frac{1}{i}\, \frac{1}{\pi\frac{z}{b_1}\sz(w_1)w_1^{-1}}.
\end{eqnarray*}
So \eqref{2.7.1} follows from \eqref{eq:phiz}. Similarly,
\[
\Res_{w=w_2} q_2\big(z\cz(-iw)\big)=\frac{1}{\pi\frac{z}{b_2}\cz(w_2)w_2^{-1}}\,.
\]
So \eqref{2.7.2} follows, too.
\end{prf}

For $j\in\{1,2\}$ and $\ell\in\Bbb Z_{\geq 0}$ we define
\begin{equation}
\label{eq:Ljl}
L_{j,\ell}=b_j(\rho_{\beta_j}+\ell)\,.
\end{equation}
Under the conditions \eqref{2.6.1} and \eqref{2.6.2} there are $\epsilon_j=\pm 1$ and $\ell_j\in\Bbb Z_{\geq 0}$ such that
\begin{eqnarray}
\label{eq:cw1-eps-k1}
z\cz(w_1)&=&i\epsilon_1 b_1(\rho_{\beta_1}+\ell_1)\ =\ i\epsilon_1 L_{1,\ell_1}\,,\\
\label{eq:cw2-eps-k2}
z\cz(-iw_2)&=&i\epsilon_2 b_2(\rho_{\beta_2}+\ell_2)\ =\ i\epsilon_2 L_{2,\ell_2}\,.
\end{eqnarray}
If $0\neq z\in \C\setminus i \big((-\infty,-L_{j,\ell}]\cup [L_{j,\ell},+\infty)\big)$, then
$\frac{i b_j}{z} (\rho_{\beta_j}+\ell)\in \C\setminus [-1,1]$.
We can therefore uniquely define $w_1^\pm, w_2^\pm \in \Dg_1\setminus \{0\}$ satisfying
\begin{eqnarray}
\label{eq:cw1pm-eps-k1}
z\cz(w_1^\pm)&=&\pm i L_{1,\ell}\,,\\
\label{eq:cw2pm-eps-k2}
z\cz(-iw_2^\pm)&=&\pm i L_{2,\ell}\,.
\end{eqnarray}
Notice that, since the domain $\Eg_{\cz(r),\sz(r)}\setminus [-1,1]$ is symmetric with respect to the
origin $0\in \C$, the element $w_j^+$ occurs in the residue sum of Lemma \ref{2.6} if and only if $w_j^-$ does.


\begin{lem}
\label{lemma:Gjkj}
 For $j\in\{1,2\}$, $\ell\in \Zb_{\geq 0}$ and
 $0\neq z\in \C\setminus i \big((-\infty,-L_{j,\ell}]\cup [L_{j,\ell},+\infty)\big)$, with the notation
 introduced above, we have
\begin{equation}
\label{eq:Resphi-psi-wpm}
\Res_{w=w_j^+} \phi_z(w)=\Res_{w=w_j^-} \phi_z(w)
\qquad \text{and} \qquad
\psi_z(w_j^+)=\psi_z(w_j^-)\,.
\end{equation}
Define
\[
G_{j,\ell}(z)=\psi_z(w_j^+) \Res_{w=w_j^+}\phi_z(w)=\psi_z(w_j^-) \Res_{w=w_j^-}\phi_z(w).
\]
Then $G_{j,\ell}$ is holomorphic on $\C\setminus i \big((-\infty,-L_{j,\ell}]\cup \{0\}\cup [L_{j,\ell},+\infty)\big)$. Explicitly, if
\begin{equation}\label{constants}
C_{j,\ell}= \tfrac{b_j}{\pi}L_{j,\ell}\, p_j(iL_{j,\ell})
\end{equation}
then $C_{j,\ell}\neq 0$ and
\begin{eqnarray}
G_{1,\ell}(z)&=&C_{1,\ell}\,\psi_z\Big(\cz^{-1}\big(\tfrac{iL_{1,\ell}}{z}\big)\Big)
p_2\Big(iz (\sz\circ \cz^{-1})\big(\tfrac{iL_{1,\ell}}{z}\big)\Big) \, 
q_2\Big(iz (\sz\circ \cz^{-1})\big(\tfrac{iL_{1,\ell}}{z}\big)\Big)\,,
\label{eq:G1k}\\
G_{2,\ell}(z)&=&C_{2,\ell}\,
\psi_z\Big(i\cz^{-1}\big(\tfrac{iL_{2,\ell}}{z}\big)\Big) 
p_1\Big(iz (\sz\circ \cz^{-1})\big(\tfrac{iL_{2,\ell}}{z}\big)\Big)\,
q_1\Big(iz(\sz\circ \cz^{-1})\big(\tfrac{iL_{2,\ell}}{z}\big)\Big)\,.
\label{eq:G2k}
\end{eqnarray}
\end{lem}
\begin{prf}
Recall from Remark \ref{rem:parity-pq} that if $m_{\beta_j}$ is odd, then the polynomial $p_j$ is even and the product $p_jq_j$ is odd.
We see from \eqref{eq:cw1pm-eps-k1} 
that $\cz(w_1^+)=-\cz(w_1^-)=\cz(-w_1^-)$. So $w_1^+=-w_1^-$. Similarly \eqref{eq:cw2pm-eps-k2}  yields $w_2^+=-w_2^-$. Thus, for $j=1,2$,
\begin{equation}
\label{eq:c-minus-i-w-pm}
\cz(w_j^-)=-\cz(w_j^+) \qquad\text{and}\qquad \cz(-iw_j^-)=-\cz(-iw_j^+)\,.
\end{equation}
The first equality in \eqref{eq:Resphi-psi-wpm} follows then from the formulas in Lemma \ref{2.7}. The second one is a consequence of the definition \eqref{eq:psiz} and the fact the spherical functions are even in the spectral parameter.

\enlargethispage*{3mm}

The explicit formulas for $G_{1,\ell}$ and $G_{2,\ell}$ are obtained from Lemma \ref{2.7}, \eqref{eq:cw1pm-eps-k1} and \eqref{eq:cw2pm-eps-k2}, as well as from
\[
\cz(-iw_1^+)=-i(\sz \circ\cz^{-1})\big(i\frac{L_{1,\ell}}{z}\big)
\qquad \text{and} \qquad
\cz(w_2^+)=i(\sz \circ\cz^{-1})\big(i\frac{L_{2,\ell}}{z}\big)\,.
\]
\end{prf}

Recall the sets $S_j$ ($j=1,2$) from \eqref{eq:Sj}.

\begin{pro}
\label{prop:SjrzW}
For $j=1,2$, for $0<r<1$ and $z\in \C\setminus i((-\infty,- L]\cup[L, +\infty))$, define
\begin{equation}
\label{eq:Sjrz}
\Sg_{j,r,z,\pm}=\{\ell\in \Zb_{\geq 0}:\pm i L_{j,\ell}\in z\big(\Eg_{\cz(r),\sz(r)}\setminus [-1,1]\big)\}\,.
\end{equation}
Let $W \subseteq \C$ be a connected open set such that
\begin{equation}
\label{eq:intersection-W}
(S_1\cup S_2) \cap W \partial \Eg_{\cz(r),\sz(r)}=\emptyset\,.
\end{equation}
Then
\begin{equation}
\label{eq:SjrzW}
S_{j,r,z,\pm}=\{\ell\in \Zb_{\geq 0}:\pm i L_{j,\ell}\in W\Eg_{\cz(r),\sz(r)}\}
\qquad (z \in W\setminus i\R)\,.
\end{equation}
and hence $S_{j,r,z,\pm}=S_{j,r,W,\pm}$ does not depend on $z \in W\setminus i\R$.
Moreover,
\begin{equation}
\label{eq:F-residues-r}
F(z)=F_r(z)+2\pi i \, G_r(z)\qquad (z \in W\setminus i\R),
\end{equation}
where
\begin{equation}
\label{eq:Gr-Sr}
G_r(z)=\sum_{\ell_1\in S_{1,r,W,+}\cup S_{1,r,W,-}} G_{1,\ell_1}(z) + \sum_{\ell_2\in S_{2,r,W,+}\cup S_{2,r,W,-}} G_{2,\ell_2}(z)\,.
\end{equation}
\end{pro}
\begin{prf}
Let $j=1,2$. The condition \eqref{eq:intersection-W} implies that for any fixed $\ell\in \Zb_{\geq 0}$ and $\epsilon\in \{\pm 1\}$, the set
$\{z\in W: \epsilon iL_{j,\ell}\in z\Eg_{\cz(r),\sz(r)}\}$ is open and closed in $W$. Since $W$ is connected, it is either $\emptyset$ or $W$. Hence, for every $z\in W$, we have
\[
\epsilon i b_j(\rho_{\beta_j}+\Zb_{\geq 0}) \cap
z\Eg_{\cz(r),\sz(r)}=
\epsilon i b_j(\rho_{\beta_j}+\Zb_{\geq 0}) \cap
W\Eg_{\cz(r),\sz(r)}\,.
\]
Moreover, if $z\in \C\setminus i\R$, then
\[
i\epsilon b_j(\rho_{\beta_j}+\Zb_{\geq 0}) \cap
z\Eg_{\cz(r),\sz(r)}=i\epsilon b_j(\rho_{\beta_j}+\Zb_{\geq 0}) \cap  z(\Eg_{\cz(r),\sz(r)}\setminus [-1,1])\,.
\]
Hence, for $z\in W\setminus i\R$, the equality \eqref{eq:SjrzW} holds and thus
$S_{j,r,z,\pm}$ depends on $W$ but not on $z$.

Notice that $w_1$ and $w_2$ respectively satisfy \eqref{eq:cw1-eps-k1} and \eqref{eq:cw2-eps-k2} if and only if
\begin{eqnarray*}
&&w_1=\cz^{-1}\big(i \epsilon_1 \tfrac{L_{1,\ell_1}}{z}\big) \quad \text{and} \quad \ell_1\in S_{1,r,z,+}\cup S_{1,r,z,-}\,,\\
&&w_2=i\cz^{-1}\big(i \epsilon_2 \tfrac{L_{2,\ell_2}}{z}\big) \quad \text{and} \quad \ell_2\in S_{2,r,z,+}\cup S_{2,r,z,-}\,.
\end{eqnarray*}
Formulas \eqref{eq:F-residues-r} and \eqref{eq:Gr-Sr} follow then from Lemmas \ref{2.6} and
\ref{lemma:Gjkj}.
\end{prf}

\begin{cor}
\label{cor:SjrvW}
For every $iv \in i\R$ and for every $r$ with $0<r<1$ and $v\c(r)\notin i\big(S_1\cup S_2\big)$
there is a connected open neighborhood $W_v$ of $iv$ in $\C$ satisfying the following conditions.
\begin{enumerate}
\item
$(S_1\cup S_2) \cap W_v \partial \Eg_{\cz(r),\sz(r)}=\emptyset$\,,
\item
$S_{j,r,W_v,+}=\{ \ell\in \Zb_{\geq 0}:i L_{j,\ell}\in iv\Eg_{\cz(r),\sz(r)}\}$\,,
\item
$S_{j,r,W_v,+}=S_{j,r,W_v,-}$\,,
\item
the equality \eqref{eq:F-residues-r} holds for $z\in W_v\setminus i\R$ with
\begin{equation}
\label{eq:Gr-Sr-i}
G_r(z)=2\sum_{\ell_1\in S_{1,r,W_v,+}} G_{1,\ell_1}(z) + 2\sum_{\ell_2\in S_{2,r,W_v,+}} G_{2,\ell_2}(z)\,.
\end{equation}
\end{enumerate}
\end{cor}
\begin{proof}
Clearly,
\begin{equation}
\label{eq:Sjr-iv}
(S_1\cup S_2) \cap iv\partial \Eg_{\cz(r),\sz(r)}=\emptyset\,.
\end{equation}
Hence, we may ``enlarge" the point $iv$ to a connected open neighborhood $W_v$ of $iv$ satisfying \eqref{eq:intersection-W}.

If $z\in W_v\setminus i\R$ is sufficiently close to $iv$, we have
\[
\{ \ell\in \Zb_{\geq 0}:i L_{j,\ell}\in iv\Eg_{\cz(r),\sz(r)}\}=
\{ \ell\in \Zb_{\geq 0}:i L_{j,\ell}\in z\Eg_{\cz(r),\sz(r)}\}\,,
\]
and, by Proposition \ref{prop:SjrzW},  the set on the right-hand side of this equality is $S_{j,r,W_v,+}$.

Observe that, since the domain  $iv\big( \Eg_{\cz(r),\sz(r)}\setminus [-1,1]\big)$ is symmetric  with respect to the origin $0\in \C$, we have $i\tfrac{L_{j,\ell}}{z} \in iv \big( \Eg_{\cz(r),\sz(r)}\setminus [-1,1]\big)$ if and only if $-i\tfrac{L_{j,\ell}}{z} \in iv \big( \Eg_{\cz(r),\sz(r)}\setminus [-1,1]\big)$. The open neighborhood $W_v$ can be chosen so that this property is preserved when $iv$ is replaced by $W_v$. Hence $S_{j,r,W_v,+} =S_{j,r,W_v,-}$. Thus \eqref{eq:Gr-Sr-i} follows from \eqref{eq:Gr-Sr}.
\end{proof}

The following lemma allows us to make explicit the sets $S_{j,r,W_v,+}$ appearing in
\eqref{eq:Gr-Sr-i}. Recall that $\floor{x}$ denotes the largest integer not greater than $x\in \R$.

\begin{lem}
\label{lem:upper-index}
Let $j=1,2$ be fixed and let $v \in \R^+$ be such that $v\geq b_j\rho_{\beta_j}$. Select an
$0<r<1$ satisfying \eqref{eq:Sjr-iv} and define
\begin{eqnarray}
\label{eq:Sjrv}
S_{j,r,v}&=&\{\ell \in \Zb_{\geq 0}; i L_{j,\ell}\in iv \Eg_{\cz(r),\sz(r)}\}  \nn \\
&=&\{\ell \in \Zb_{\geq 0}; \rho_{\beta_j}+\ell\in \nu_j \Eg_{\cz(r),\sz(r)}\}\,,
\end{eqnarray}
where $\nu_j=\frac{v}{b_j}$.
Then the following properties hold.
\begin{enumerate}
\item
$S_{j,r,v}=\{0,1,\dots, N_{j,v}\}$ for some $N_{j,v}\in \Zb_{\geq 0}$.
\item
If $r$ is chosen so that
\begin{equation}
\label{eq:first condition on r}
\cz(r)< \frac{\floor{\nu_j}+1}{\nu_j}
\end{equation}
(or even $\cz(r)< \frac{\floor{\nu_j}+\frac{3}{2}}{\nu_j}$ if $\rho_{\beta_j} \in \Zb+\frac{1}{2}$),  then
$N_{j,v}\leq \floor{\nu_j}-\floor{\rho_{\beta_j}}$.
\item
Suppose that $\rho_{\beta_j} \in \Zb$. Then $N_{j,v}\geq \floor{\nu_j}-\rho_{\beta_j}$.
Thus
\[
N_{j,v}= \floor{\nu_j}-\rho_{\beta_j}
\]
if \eqref{eq:first condition on r} is satisfied.
\item
Suppose that $\rho_{\beta_j} \in \Zb+\frac{1}{2}$. 
\newline
\noindent
If $\floor{\nu_j}+\frac{1}{2}\leq\nu_j$, 
then $N_{j,v}\geq  \floor{\nu_j}-\floor{\rho_{\beta_j}}$.
Hence
\[
N_{j,v}= \floor{\nu_j}-\floor{\rho_{\beta_j}}
\]
if \eqref{eq:first condition on r} is satisfied.
\newline
\noindent
If $\floor{\nu_j}+\frac{1}{2}>\nu_j$ and
$\cz(r)< \frac{\floor{\nu_j}+\frac{1}{2}}{\nu_j}$, then
\[
N_{j,v}= \floor{\nu_j}-\floor{\rho_{\beta_j}}-1\,.
\]
\end{enumerate}
\end{lem}
\begin{prf}
To simplify notation, in this proof we shall omit the indices $j$.

Notice first that, for $u\in \R^+$, we have that $iu \in iv \Eg_{\cz(r),\sz(r)}$ if and only if $u<\cz(r) v$, and $iu \notin iv \overline{\Eg_{\cz(r),\sz(r)}}$ if and only if $u>\cz(r) v$.

In particular, since $v\geq b\rho_\beta$ and $\cz(r)\geq 1$, we have
$\cz(r) \nu > \nu \geq \rho_\beta$. Hence $0\in S_{j,r,v}$. Since
$iv \Eg_{\cz(r),\sz(r)}$ is a bounded convex set containing $0\in \C$, the set $S_{j,r,v}$ is of the
form given in (1) for some $N_v\in \Zb_{\geq 0}$.

Observe also that $v\geq b\rho_\beta$ implies $\nu\geq \rho_\beta\geq \floor{\rho_\beta}$.
So $\floor{\nu}\geq \floor{\rho_\beta}$ and $\floor{\nu}-\floor{\rho_\beta}\in \Zb_{\geq 0}$.

The inequality $N_{v}\leq \floor{\nu}-\floor{\rho_{\beta}}$, i.e. $\floor{\nu}-\floor{\rho_{\beta}}+1 \notin S_{j,r,v}$, is equivalent to
$\rho_{\beta}+(\floor{\nu}-\floor{\rho_{\beta}}+1)\notin \nu \Eg_{\cz(r),\sz(r)}$, i.e.
$\rho_{\beta}+(\floor{\nu}-\floor{\rho_{\beta}}+1)\geq \nu \cz(r)$.

We have
\[
\rho_{\beta}+(\floor{\nu}-\floor{\rho_{\beta}}+1)=
\begin{cases}
\floor{\nu}+\frac{1}{2}+1 &\text{if $\rho_\beta\in \Zb+\frac{1}{2}$}\\
\floor{\nu}+1 &\text{if $\rho_\beta\in \Zb$}\,.
\end{cases}
\]
In both cases,
\[
\rho_{\beta}+(\floor{\nu}-\floor{\rho_\beta}+1)\geq \floor{\nu}+1 > \cz(r_v) \nu
\]
if \eqref{eq:first condition on r} holds. This proves (2).

Suppose now that $\rho_\beta\in \Zb$. Then
\[
\rho_{\beta}+(\floor{\nu}-\rho_\beta)=\floor{\nu}\leq \nu < \cz(r) \nu\,.
\]
Hence $\floor{\nu}-\rho_{\beta}\in S_{j,r,v}$, i.e. $N_v\geq \floor{\nu}-\rho_{\beta}$.
The equality (3) follows then from (2).

Finally, suppose $\rho_\beta\in \Zb+\frac{1}{2}$.

If $\floor{\nu}+\frac{1}{2}\leq \nu$, then
\[
\rho_{\beta}+(\floor{\nu}-\floor{\rho_\beta})=\floor{\nu}+\frac{1}{2}\leq \nu < \cz(r) \nu\,.
\]
Hence $\floor{\nu}-\floor{\rho_{\beta}}\in S_{j,r,v}$ and the first claim of (4)
follows then from (2).

If $\floor{\nu}+\frac{1}{2}> \nu$ and
$\cz(r)< \frac{\floor{\nu_j}+\frac{1}{2}}{\nu_j}$. Then
\[
\rho_{\beta}+(\floor{\nu}-\floor{\rho_\beta})=\floor{\nu}+\frac{1}{2}> \cz(r) \nu\,.
\]
Hence $\floor{\nu}-\floor{\rho_{\beta}} > N_v$.
On the other hand,
\[
\rho_{\beta}+(\floor{\nu}-\floor{\rho_\beta}-1)=\floor{\nu}-\frac{1}{2}< \nu <\cz(r) \nu\,.
\]
So $\floor{\nu}-\floor{\rho_{\beta}}-1\geq N_v$, and the equality follows.
\end{prf}

\begin{cor}\label{cor-SjrvW'}
Keep the notation introduced above, and set for $j\in \{1,2\}$ and $m\in \Zb_{\geq 0}$:
\begin{equation}
\label{eq:Ijm}
I_{j,m}=b_j\rho_{\beta_j}+b_j[m,m+1)=[L_{j,m},L_{j,m+1})\,.
\end{equation}
Suppose that $v \in I_{j,m}$. 
Let $r_m$ be such that $0<r_m<1$ and $c(r_m)=\frac{\rho_{\beta_j}+m +\frac{1}{2}}{\rho_{\beta_j}+m}$.
Then
\[
\sum_{\ell\in S_{j,r,W_v,+}} G_{j,\ell}(z)=\sum_{\ell=0}^m G_{j,\ell}(z) \qquad (r_m<r<1,\ z\in W_v\setminus i\R)\,.
\]
\end{cor}
\begin{proof}
This is a consequence of (2) in Corollary \ref{cor:SjrvW} and Lemma \ref{lem:upper-index}. Indeed, suppose first that
$\rho_{\beta_j} \in \Zb$ and $\nu_j=\frac{v}{b_j}\in \rho_{\beta_j}+[m,m+1)$. Then, with the given choice of $r$, $N_{j,v}=\floor{\nu_j}-\rho_{\beta_j}=(m+\rho_{\beta_j})-\rho_{\beta_j}=m$.

Similarly, suppose now that $\rho_{\beta_j} \in \Zb+\frac{1}{2}$.
If $\nu_j\in \rho_{\beta_j}+\big[m,m+\frac{1}{2}\big)$, then 
$N_{j,v}=\floor{\nu_j}-\floor{\rho_{\beta_j}}=(m+\floor{\rho_{\beta_j}})-\floor{\rho_{\beta_j}}=m$.
If $\nu_j\in \rho_{\beta_j}+\big[m+\frac{1}{2},m+1\big)$, 
then $N_{j,v}=\floor{\nu_j}-\floor{\rho_{\beta_j}}-1=(\floor{\rho_{\beta_j}}+m+1)-
\floor{\rho_{\beta_j}}-1=m$.
\end{proof}

\subsection{The case of one odd multiplicity}
\label{subsection:contours-even-mul}

In this subsection we assume that only one of the two multiplicities $m_{\beta_1}$ and $m_{\beta_2}$ is odd. (Recall that the case where both multiplicities are even is covered by \cite[Theorem 3.3(2)]{Str05}. See also Proposition \ref{pro:holoextRFnegaxis} above). 

Without loss of generality, we shall assume that $m_{\beta_2}$ is even. Hence
\[
\rho_{\beta_2}\in \Zb \quad \text{and} \quad L=b_2\rho_{\beta_2}=\sqrt{\langle\rho_2,
\rho_2\rangle}\,.
\]
Moreover, the Plancherel density \eqref{eq:Plancherel-density} for the symmetric space $\X_2$ is a polynomial because $q_2$ is the constant function $1$. Therefore, the second sum on the right-hand side of the displayed formula in Lemma \ref{2.6} does not appear. The computations for the local meromorphic extension of the function $F$ near the imaginary axis, done in subsection \ref{subsection:contours-odd-mul}, apply, in a simplified version, to the case of $m_{\beta_1}$ even as well. The following proposition summarizes the results.
Recall from \eqref{eq:Sj}  that $S_1=ib_1 \big((\rho_{m_1}+\Zb_{\geq 0})\cup (-\rho_{m_1}-\Zb_{\geq 0})\big)$. Moreover, as in \eqref{constants}, 
\[
C_{1,\ell}= \tfrac{b_1}{\pi}L_{1,\ell}\, p_1(iL_{1,\ell}) \quad\text{with} \quad 
L_{1,\ell}=b_1(\rho_{\beta_1}+\ell)\,.
\]

\begin{pro}
\label{pro:local-ext-F-even}
For every $iv \in i\R$ and for $r$, with $0<r<1$ and $vc(r)\notin S_1$, there is a connected
open neighborhood $W_v$ of $iv$ in $\C$ satisfying 
\[
S_1\cap W_v \partial \Eg_{\cz(r),\sz(r)}=\emptyset\,,
\]
so that for $z\in W_v\setminus i\R$ we have 
\[
F(z)=F_r(z)+2\pi i G_r(z)\,,
\]
where $F_r(z)$ and $G_r(z)$ are as in Proposition \ref{2.5}. 

If $v\geq b_1\rho_{\beta_1}$ and $v\in I_{1,m}=b_1\rho_{\beta_1}  +b_1[m,m+1)=[L_{1,m},L_{1,m+1})$ with 
$m\in \Zb_{\geq 0}$ and $c(r)<\frac{\rho_{\beta_j}+m +\frac{1}{2}}{\rho_{\beta_j}+m}$, then 
\begin{equation}
\label{eq:Gr-even}
G_r(z)=2\sum_{\ell=0}^m G_{1,\ell}(z)\,,
\end{equation}
where
\begin{equation}
\label{eq:G1l-even}
G_{1,\ell}(z)=C_{1,\ell}\,\psi_z\Big(\cz^{-1}\big(\tfrac{iL_{1,\ell}}{z}\big)\Big)
p_2\Big(z (\sz\circ \cz^{-1})\big(\tfrac{iL_{1,\ell}}{z}\big)\Big)\,
\end{equation}
and $C_{1,\ell}\neq 0$.
\end{pro}

\section{Odd multiplicities: meromorphic extensions}
\label{section:meroext-odd}

In this section we determine the meromorphic extension of the resolvent in the case in which both multiplicities $m_{\beta_1}$ and $m_{\beta_2}$ are odd. 

\subsection{Extension of the functions  $G_{j,\ell}$}
\label{subsection:meroextGj}

Fix $j\in\{1,2\}$ and $\ell\in\Bbb Z_{\geq 0}$. Then
\begin{equation}
\label{eq:Mjk}
\M_{j,\ell}=\Big\{(z,\zeta)\in\C^\times\times(\C\setminus\{i, -i\})\ :\ \zeta^2=\Big(\frac{iL_{j,\ell}}{z}\Big)^2-1\Big\}
\end{equation}
is a Riemann surface above $\C^\times$, with projection map $\pi_{j,\ell}:\M_{j,\ell} \ni (z,\zeta)\to z\in \C^\times$\,.
The fiber of $\pi_{j,\ell}$ above $z\in \C^\times$ is $\{(z,\zeta), (z,-\zeta)\}$.
In particular, the restriction of $\pi_{j,\ell}$ to $\M_{j,\ell}\setminus \{(\pm i L_{j,\ell},0)\}$ is a double cover
of $\C^\times \setminus \{\pm i L_{j,\ell}\}$.

\begin{lem}\label{2.8}
Suppose that $m_{\beta_1}$ and $m_{\beta_2}$ are odd. Then, in the above notation,
\begin{eqnarray}\label{eq:tG1k}
\wt G_{1,\ell}:\M_{1,\ell}\ni (z,\zeta)&\to&
\frac{b_1}{\pi} \, L_{1,\ell}\,  p_1(i L_{1,\ell})\psi_z\big(\frac{i L_{1,\ell}}{z}-\zeta\big)
 p_2(-iz\zeta) q_2(-iz\zeta)\in\C
\end{eqnarray}
is the meromorphic extension to $\M_{1,\ell}$ of a lift of $G_{1,\ell}$, and
\begin{eqnarray}\label{eq:tG2k}
\wt G_{2,\ell}:\M_{2,\ell}\ni (z,\zeta)&\to& \frac{b_2}{\pi}\, L_{2,\ell}\,
p_2(i L_{2,\ell}) \psi_z\big(i\big(\frac{iL_{2,\ell}}{z}-\zeta\big)\big)
 p_1(-iz\zeta) q_1(-iz\zeta)\in\C
\end{eqnarray}
is the meromorphic extension to $\M_{2,\ell}$ of a lift of $G_{2,\ell}$.

The function $\wt G_{1,\ell}$ has simple poles at all points $(z,\zeta)\in \M_{1,\ell}$ such that
\begin{equation}\label{2.9.1}
z=\pm i \sqrt{ L_{1,\ell}^2+L_{2,m}^2},
\end{equation}
where $m\in\Zb_{\geq 0}$.
The function $\wt G_{2,\ell}$ has simple poles at all points $(z,\zeta)\in \M_{2,\ell}$ such that
\begin{equation}\label{2.10.1}
z=\pm i \sqrt{L_{2,\ell}^2+L_{1,m}^2},
\end{equation}
where $m\in\Zb_{\geq 0}$.
\end{lem}
\begin{prf}
The map
\[
\sigma_{j,\ell}^+:\C\setminus i \big(-\infty, -L_{j,\ell}]\cup [L_{j,\ell},+\infty)\big) \ni z \to (z,\zeta_{j,\ell}^+(z)) \in \M_{j,\ell}\,,
\]
where
\begin{equation}
\label{eq:zeta+}
\zeta_{j,\ell}^+(z)=\sqrt{\frac{iL_{j,\ell}}{z}+1}\,
\sqrt{\frac{iL_{j,\ell}}{z}-1}\,,
\end{equation}
is a holomorphic section of  $\pi_{j,\ell}$ and, on the domain of $\sigma_{j,\ell}^+$\,,
\begin{equation}
\label{eq:wtGsectionjl}
\wt G_{j,\ell} \circ \sigma_{j,\ell}^+ =G_{j,\ell}\,.
\end{equation}
by \eqref{eq:cinverse}, \eqref{eq:scinverse},  \eqref{constants}, \eqref{eq:G1k} and \eqref{eq:G2k}.
So, $\wt G_{j,\ell}$ is a lift of $G_{j,\ell}$. Its meromorphic extension follows from \eqref{eq:lift of c} and \eqref{eq:lift of sc}.

The poles of $\wt G_{1,\ell}$ are the points $(z,\zeta)\in \M_{1,\ell}$ such that
\begin{equation}
\label{eq:p2q2singular}
p_2(-iz\zeta)\, q_2(-iz\zeta)
\end{equation}
is singular (and hence simple poles). They are such that there is an $m\in\Bbb Z_{\geq 0}$ and $\epsilon_2=\pm 1$ with
\begin{equation}\label{2.9.2}
-iz\zeta=\epsilon_2 i L_{2,m}.
\end{equation}
By definition of $\M_{1,\ell}$,
\begin{equation}\label{2.9.3}
\zeta^2=-\frac{L_{1,\ell}^2}{z^2}-1.
\end{equation}
Formula \eqref{2.9.1}  is equivalent to
\begin{equation*}\label{2.9.4}
-z^2=L_{1,\ell}^2+L_{2,m}^2,\qquad z\ne 0.
\end{equation*}
So $z\in i\R$ and, by  \eqref{2.9.3},
\begin{equation}\label{2.9.5}
\zeta^2=\frac{L_{1,\ell}^2}{L_{1,\ell}^2+L_{2,m}^2}\,-1=
   -\frac{L_{2,m}^2}{L_{1,\ell}^2+L_{2,m}^2}<0.
\end{equation}
Therefore $\zeta\in i\R$ and $(z,\zeta)$ satisfies \eqref{2.9.2}.

Conversely, one can easily check that for every $m\in \Zb_{\geq 0}$, the element
$(z,\zeta)$ defined by \eqref{2.9.1} and \eqref{2.9.2} is in $\M_{1,\ell}$ and makes
\eqref{eq:p2q2singular} (and hence $\wt G_{1,\ell}$) singular.

The computation of the singularities of $\wt G_{2,\ell}$ is similar.
\end{prf}

For $\ell_1, \ell_2\in \Zb_{\geq 0}$, set
\begin{equation}
\label{eq:z-pmk1k2}
z_{\ell_1,\ell_2}=i \sqrt{L_{1,\ell_1}^2+L_{2,\ell_2}^2}
\end{equation}
and
\begin{equation}\label{2.9.5'}
\zeta_{j,\ell_1,\ell_2}= i \sqrt{\frac{L_{3-j,\ell_{3-j}}^2}{L_{1,\ell_1}^2+L_{2,\ell_2}^2}}\,,
\end{equation}
where $j\in \{1,2\}$. Then 
the points of the fiber of $\pm z_{\ell_1,\ell_2}$ in $\M_{j,\ell_j}$ are $(\pm z_{\ell_1,\ell_2},\epsilon \zeta_{j,\ell_1,\ell_2})$ with $\epsilon=\pm 1$. 
For fixed $j\in \{1,2\}$ and $\ell_j\in \Zb_{\geq 0}$ the set
\begin{equation}
\label{eq:U1pm}
\Ug_{j,\ell_j,\pm}=\{(z,\zeta)\in \M_{j,\ell_j}\ ;\ \pm\Im z >0\}
\end{equation}
is an open neighborhood of all the points
$(\pm z_{\ell_1,\ell_2},\epsilon\zeta_{j,\ell_1,\ell_2})\in \M_{j,\ell_j}$.
Moreover, the following maps are local charts:
\begin{equation}\label{charts}
\kappa_{j,\ell_j,\pm}:\Ug_{j,\ell_j,\pm}\ni (z,\zeta)\to \zeta\in \C\setminus  i\big((-\infty, -1] \cup [1,+\infty)\big),\ \ z=\pm i \,\frac{L_{j,\ell_j}}{\sqrt{\zeta^2+1}}.
\end{equation}

\begin{lem}\label{lem:chart-expressions}
For $j=1,2$ the local expressions for $\wt G_{j,\ell_j}$ in terms of the charts (\ref{charts}) are
\begin{eqnarray}
\label{expression in terms of charts 1}
\big(\wt G_{1,\ell_1}\circ\kappa_{1,\ell_1,\pm}^{-1}\big)(\zeta) 
=\pm \frac{b_1}{\pi} \, L_{1,\ell_1}\,  p_1(i L_{1,\ell_1})
p_2\Big(\tfrac{L_{1,\ell_1}\zeta}{\sqrt{\zeta^2+1}}\Big) q_2\Big(\tfrac{L_{1,\ell_1}\zeta}{\sqrt{\zeta^2+1}}\Big)
\psi_{i\,\frac{L_{1,\ell_1}}{\sqrt{\zeta^2+1}}}\big(\sqrt{\zeta^2+1}\mp \zeta\big)
\end{eqnarray}
and
\begin{eqnarray}\label{expression in terms of charts 2}
\big(\wt G_{2,\ell_2}\circ\kappa_{2,\ell_2,\pm}^{-1}\big)(\zeta)
=\pm \frac{b_2}{\pi} \, L_{2,\ell_2}\,  p_2(i L_{2,\ell_2})
p_1\Big(\tfrac{L_{2,\ell_2}\zeta}{\sqrt{\zeta^2+1}}\Big) q_1\Big(\tfrac{L_{2,\ell_2}\zeta}{\sqrt{\zeta^2+1}}\Big)
\psi_{i\,\frac{L_{2,\ell_2}}{\sqrt{\zeta^2+1}}}
\big( i \big(\sqrt{\zeta^2+1}\mp \zeta\big)\big)\,.
\end{eqnarray}
With the above notation, the residue of the local expression of $\wt G_{1,\ell_1}$ at a point $(z,\zeta)\in \M_{1,\ell_1}$ with $z=\pm z_{\ell_1,\ell_2}$ is
\begin{equation}\label{2.11.1}
\Res_{\zeta=\pm \zeta_{1,\ell_1,\ell_2}} (\wt G_{1,\ell_1}\circ \kappa_{1,\ell_1,\pm}^{-1})(\zeta)
=\pm \frac{b_1}{i\pi^2} \, L_{1,\ell_1}\,  p_1(i L_{1,\ell_1}) p_2(i L_{2,\ell_2})
(f\times \varphi_{(\rho_{\beta_1}+\ell_1)\beta_1+(\rho_{\beta_2}+\ell_2)\beta_2})(y)\,.
\end{equation}
The residue of the local expression of $\wt G_{2,\ell_2}$ at $(z,\zeta)\in \M_{2,\ell_2}$ with $z=\pm z_{\ell_1,\ell_2}$ is
\begin{equation}\label{2.11.2}
\Res_{\zeta=\pm \zeta_{2,\ell_1,\ell_2}} (\wt G_{2,\ell_2}\circ \kappa_{2,\ell_2,\pm}^{-1})(\zeta)=
\pm \frac{b_2}{i\pi^2} \, L_{2,\ell_2}\,  p_1(i L_{1,\ell_1})p_2(i L_{2,\ell_2})
(f\times \varphi_{(\rho_{\beta_1}+\ell_1)\beta_1+(\rho_{\beta_2}+\ell_2)\beta_2})(y).
\end{equation}
The constants
\begin{eqnarray*}
&& b_1 L_{1,\ell_1}\,  p_1(i L_{1,\ell_1}) p_2(i L_{2,\ell_2})\\
&& b_2 L_{2,\ell_2}\,  p_1(i L_{1,\ell_1}) p_2(i L_{2,\ell_2})
\end{eqnarray*}
appearing in the above formulas are positive.
\end{lem}
\begin{prf}
Recall that
\[
\varphi_{x_1\beta_1+x_2\beta_2}=\varphi_{\pm x_1\beta_1\pm x_2\beta_2}.
\]
and that the polynomials $p_1$ and $p_2$ are even when $m_{\beta_1}$ and $m_{\beta_2}$
are odd.

Let $(z,\zeta)\in \M_{1,\ell_1}$ with 
$z=\pm z_{\ell_1,\ell_2}$ with $\ell_2\in \Zb_{\geq 0}$.
Then, as in \eqref{2.9.2} with $m=\ell_2$, we have $-iz\zeta=i\epsilon_2 L_{2,\ell_2}$ for 
$\epsilon_2\in \{\pm 1\}$,  and
\[
\Res_{x=i\epsilon_2 L_{2,\ell_2}} q_2(x)=
 \Res_{x=i\epsilon_2L_{2,\ell_2}} \, \frac{1}{i} \, \coth(\pi(\frac{x}{b_2}+i\rho_{\beta_2}))=
\frac{1}{i\pi}\,.
\]

We see from \eqref{eq:tG1k} and \eqref{expression in terms of charts 1} that the residue of $\wt G_{1,\ell_1}$ at $(z,\zeta)$ is
\[
\pm\frac{b_1}{i\pi^2} \, L_{1,\ell_1}\, p_1(i L_{1,\ell_1})\,  p_2(i \epsilon_2 L_{2,\ell_2})
\psi_z(w)\,,
\]
where
\[
\psi_z(w)=(f\times \varphi_{\bi\frac{z}{b_1}\cz(w)\,\beta_1+\bi\frac{z}{b_2}\cz(-iw)\,\beta_2})(y)\,,
\]
with
\[
w=\frac{i L_{1,\ell_2}}{z}-\zeta=\frac{i L_{1,\ell_2}+\epsilon_2  L_{2,\ell_2}}{\pm i \sqrt{ L_{1,\ell_2}^2+ L_{2,\ell_2}^2}}\,.
\]
Observe that
\[
w^{-1}=\frac{\pm i \sqrt{ L_{1,\ell_2}^2+ L_{2,\ell_2}^2}}{i L_{1,\ell_2}+\epsilon_2  L_{2,\ell_2}}=
\frac{i L_{1,\ell_2}-\epsilon_2  L_{2,\ell_2}}{\pm i \sqrt{ L_{1,\ell_2}^2+ L_{2,\ell_2}^2}}\,.
\]
So
\begin{eqnarray*}
&&\cz(w)=\frac{w+w^{-1}}{2}=\pm \frac{ L_{1,\ell_2}}{\sqrt{ L_{1,\ell_2}^2+ L_{2,\ell_2}^2}}\,,\\
&&\cz(-iw)=\frac{w-w^{-1}}{2i}=\mp \frac{\epsilon_2  L_{2,\ell_2}}{\sqrt{ L_{1,\ell_2}^2+ L_{2,\ell_2}^2}}\,,\\
\end{eqnarray*}
and hence
\[
z\, \cz(w)=i L_{1,\ell_1} \qquad \text{and} \qquad
z\, \cz(-iw)=-i\epsilon_2 L_{2,\ell_2}\,.
\]
Since $\bi\, i=-1$, this verifies \eqref{2.11.1}.

The proof of \eqref{2.11.2} is similar, using
\[
\psi_z(iw)=(f\times \varphi_{\bi\frac{z}{b_1}\cz(-iw)\,\beta_1+\bi\frac{z}{b_2}\cz(w)\,\beta_2})(y)
\]
where
\[
w=i\frac{L_{2,\ell_1}}{z}-\zeta
\qquad \text{and} \qquad
z\zeta=-\epsilon_1 L_{1,\ell_2}
\]
with $\epsilon_1=\pm 1$.

For the positivity of the constants in \eqref{2.11.1} and \eqref{2.11.2}, notice that, omitting indices as in the notation from section \ref{subsection:rankone}, we have $p(ibx)=P(-x)=P(x)$
since $P$ is even when $m_\beta$ is odd.  If $\ell\in \Zb_{\geq 0}$, then $\rho_\beta+\ell$
is bigger than all roots of $P$, and $P$ is monic. Thus $P(\rho_\beta+\ell)>0$.

\end{prf}

\subsection{Piecewise extension of $F$}
\label{piecewise}

For $j\in\{1,2\}$ set $I_{j,-1}=(0,b_j\rho_{\beta_j})$
and, for $x\geq 0$, 
\begin{equation}
\label{eq:Ijx}
I_{j,x}=b_j\rho_{\beta_j}+b_j[x,x+1)\,.
\end{equation}
Notice that if $x=m\in \Zb_{\geq 0}$, then $I_{j,m}=[L_{j,m},L_{j,m+1})$. So  \eqref{eq:Ijx} extends the notation introduced in \eqref{eq:Ijm}.  (We shall need it for $x \in \Zb_{\geq 0}$ and 
$x \in -\frac{1}{2}+\Zb_{\geq 0}$.)
We also define 
\[
L_{j,-1}=0 \qquad \text{and} \qquad I_{j,-1}=(0,L_{j,0})\,.
\]

Let  $v \in (0,+\infty)$. Then there exists a unique pair $(m_1,m_2)\in \Zb_{\ge -1}\times \Zb_{\ge -1}$ 
such that $v\in I_{1,m_1}\cap  I_{2,m_2}$. Suppose first that $(m_1,m_2)\neq (-1,-1)$. 
Then, according to Corollaries~\ref{cor:SjrvW} and \ref{cor-SjrvW'}, there exists $0< r_v <1$ and an open neighborhood $W_v$ of $-iv$ in $\C$ such that 
\begin{equation}
\label{eq:Fresidues1}
F(z)=F_{r_v}(z)+ 2\sum_{\ell_1=0}^{m_1} G_{1,\ell_1}(z) + 2\sum_{\ell_2=0}^{m_2} G_{2,\ell_2}(z)
\qquad (z\in W_v\setminus i\R)\, ,
\end{equation}
where the function $F_{r_v}$ is holomorphic in $W_v$ and, as usual, empty sums are equal to $0$.
We can extend  \eqref{eq:Fresidues1} to $v\in I_{1,-1}\cap  I_{2,-1}=(0,L=\min\{L_{1,0},L_{2,0}\})$ by choosing $W_v$ to be an open disk centered at $-iv$ 
whose intersection with $\R$ is entirely contained in 
$I_{1,-1}\cap  I_{2,-1}$ and by setting $F_{r_v}(z)=F(z)$ for $z \in W_v$.

Notice that for $(m_1,m_2)\in \Zb_{\ge -1}\times \Zb_{\ge -1}\setminus \{(-1,-1)\}$ the left boundary point of $ I_{1,m_1}\cap  I_{2,m_2}$, if this interval is nonempty, is $\max\{L_{1,m_1}, L_{2,m_2}\}$. Moreover, $L_{1,m_1}$ and $L_{2,m_2}$ can be equal. 

Let $m\in \Zb_{\geq 0}$. Denote by $I^\circ$ the interior of the set $I$. By shrinking $W_v$ if necessary, we may assume that $W_v$ is an open disk centered at $-iv$ such that
for $j\in \{1,2\}$ we have
\begin{equation}
\label{eq:WviR}
W_v\cap i\R \subseteq
\begin{cases}
  -i I_{j,m}^\circ &\text{if $v\in I_{j,m}^\circ$},\\
  -i I_{j,m-\frac{1}{2}}^\circ &\text{if $v=L_{j,m}$}.
\end{cases}
\end{equation}

Suppose $W_v\cap W_{v'}\not=\emptyset$. By \eqref{eq:WviR}, if $v\in I_{j,m}^\circ$ then 
$v'\in I_{j,m}^\circ$, and if $v=L_{j,m}$ then $v'\in I_{j,m-1}^\circ \cup  I_{j,m}$. Only the first case can occur if $m=-1$.

So, let $(m_1,m_2)\in \Zb_{\ge -1}\times \Zb_{\ge -1}$. 
Comparing \eqref{eq:Fresidues1} for $v\in I_{1,m_1}\cap I_{2,m_2}$ and for $v'\in (0,+\infty)$ with 
$W_v\cap W_{v'}\not=\emptyset$, we obtain for $z\in W_v\cap W_{v'}\setminus i\R$
\begin{multline}
\label{eq:Frvv'}
F_{r_{v'}}(z)=F_{r_v}(z) \\ +
\begin{cases}
  0 &\text{if $v'\in I_{1,m_1}\cap I_{2,m_2}$},\\
  2\, G_{1,m_1}(z) &\text{if $v=L_{1,m_1}\in I_{2,m_2}^\circ$ and $v'\in I_{1,m_1-1}^\circ\cap I_{2,m_2}$},\\
  2\, G_{2,m_2}(z) &\text{if $v=L_{2,m_2}\in I_{1,m_1}^\circ$ and $v'\in I_{1,m_1}\cap I_{2,m_2-1}^\circ$},\\
  2\, G_{1,m_1}(z)+2\, G_{2,m_2}(z) &\text{if $v=L_{1,m_1}=L_{2,m_2}$ and $v'\in I_{1,m_1-1}^\circ\cap I_{2,m_2-1}^\circ$}\,.
\end{cases}
\end{multline}
Notice that, by the above discussion, these four cases cover all the non-overlapping possibilities. 

Now we set for $(m_1,m_2)\in \Zb_{\geq -1}\times \Zb_{\geq -1}$
\[
W_{(m_1,m_2)}=\bigcup_{v\in I_{1,m_1}\cap I_{2,m_2}} W_v\,.
\]
Moreover, for all $(m_1,m_2)$ with $I_{1,m_1}\cap I_{2,m_2}\neq \emptyset$, 
 we define a function $F_{(m_1,m_2)}$ on $W_{(m_1,m_2)}$ by 
\[ 
 F_{(m_1,m_2)}(z)= F_{r_v}(z) \qquad \text{if $v\in I_{1,m_1}\cap I_{2,m_2}$ and $z\in W_v$}.
\] 
The first case of \eqref{eq:Frvv'} ensures that  $F_{(m_1,m_2)}: W_{(m_1,m_2)}\to \C$ 
is holomorphic. 

We collect the above considerations in the following proposition (cf. \cite[Cor.18]{HPP14}).

\begin{pro}
\label{cor:FextendedMainSection}
For every integers $(m_1,m_2)\in \Zb_{\geq -1}\times \Zb_{\geq -1}$ for which $I_{1,m_1}\cap I_{2,m_2}\neq \emptyset$, we have
 \begin{equation}
\label{eq:Fresidues2}
F(z)=F_{(m_1,m_2)}(z)+2\sum_{\ell_1=0}^{m_1} G_{1,\ell_1}(z)+2\sum_{\ell_2=0}^{m_2} G_{2,\ell_2}(z)
\qquad (z\in W_{(m_1,m_2)}\setminus i\R)\,,
\end{equation}
where $F_{(m_1,m_2)}$ is holomorphic in $W_{(m_1,m_2)}$, the $G_{j,\ell_j}$ are as in Lemma~\ref{lemma:Gjkj}, and empty sums are defined to be equal to $0$.
\end{pro}

It will be more convenient to use a different parametrization for the intervals partitionning $-i(0,+\infty)$. Let $\{L_\ell\}_{\ell=0}^{\infty}$ be the set of the distinct 
$L_{j,\ell_j}$, where $j\in\{1,2\}$ and $\ell_j\in \Zb_{\geq 0}$, ordered according to their distance from $0$. Notice that $L_0=L$. We also set $L_{-1}=0$. 

Let  $(m_1,m_2)\in \Zb_{\geq -1}\times\Zb_{\geq -1}$. Then $I_{1,m_1}\cap I_{2,m_2}\neq \emptyset$ if and only if there is $\tau(m_1,m_2)\in \Zb_{\geq 0}$ such that $(I_{1,m_1}\cap I_{2,m_2})^\circ=
(L_{\tau(m_1,m_2)}, L_{\tau(m_1,m_2)+1})$. In this case, $L_{\tau(m_1,m_2)}=\max\{L_{1,m_1},L_{2,m_2}\}$. The correspondence $(m_1,m_2)\to 
\tau(m_1,m_2)$ is a bijection between the set of pairs  $(m_1,m_2)\in \Zb_{\geq -1}\times \Zb_{\geq -1}$ for which $I_{1,m_1}\cap I_{2,m_2}\neq \emptyset$ and the set  $\Zb_{\geq -1}$. 
We define
\begin{equation*}
W_{(\tau(m_1,m_2))}=W_{(m_1,m_2)} \quad\text{and}\quad  F_{(\tau(m_1,m_2))}=F_{(m_1,m_2)}\,.
\end{equation*}

Furthermore, for all $\ell \in \Zb_{\geq -1}$, we set
\begin{equation}
\label{eq:Gell}
G_{(\ell)}=\begin{cases}
G_{j,\ell_j} &\text{if $L_\ell=L_{j,\ell_j}$ for a unique pair $(j,\ell_j)$\,,}\\ 
G_{1,\ell_1}+G_{2,\ell_2}  &\text{if $L_\ell=L_{1,\ell_1}=L_{2,\ell_2}$ for some $\ell_1,\ell_2$}\,.
\end{cases} 
\end{equation} 
Then we have the following restatement of Proposition \ref{cor:FextendedMainSection}.

\begin{pro}
\label{pro:FextendedMainSection-Z}
For every integer $m\in \Zb_{\ge -1}$, we have
 \begin{equation}
\label{eq:Fresidues3}
F(z)=F_{(m)}(z)+2\sum_{\ell=0}^{m} G_{(\ell)}(z)
\qquad (z\in W_{(m)}\setminus i\R)\,,
\end{equation}
where $F_{(m)}$ is holomorphic in $W_{(m)}$, the $G_{(\ell)}$ are as in \eqref{eq:Gell}, and empty sums are defined to be equal to $0$.
\end{pro}

Recall the Riemann surface $\M_{j,\ell}$ to which the fixed lift of $G_{j,\ell}$ extends as a meromorphic function. To unify the notation, we will write $\M_\ell$ instead of $\M_{j,\ell_j}$ if $L_\ell=L_{j,\ell_j}$. Notice that this is well defined when $L_\ell=L_{1,\ell_1}=L_{2,\ell_2}$ for some $\ell_1,\ell_2$. Likewise, the projection map and the fixed chart map will be written respectively $\pi_\ell$ and $\kappa_\ell$ instead of $\pi_{j,\ell_j}$ and $\kappa_{j,\ell_j}$.
\label{notation:ell}

By \eqref{eq:Gell}, the function $G_{(\ell)}$ can be lifted and extended to a meromorphic function $\wt G_{(\ell)}$ on $\M_\ell$. Specifically, 
\begin{equation}
\label{eq:wtGell}
\wt G_{(\ell)}=\begin{cases}
\wt G_{j,\ell_j} &\text{if $L_\ell=L_{j,\ell_j}$ for a unique pair $(j,\ell_j)$\,,}\\ 
\wt G_{1,\ell_1}+\wt G_{2,\ell_2}  &\text{if $L_\ell=L_{1,\ell_1}=L_{2,\ell_2}$ for some $\ell_1,\ell_2$}\,.
\end{cases} 
\end{equation} 
The singularities of $\wt G_{(\ell)}$ are at most simple poles located at points $(z,\zeta)\in \M_\ell$
where 
\begin{equation}
\label{eq:sing-wtGell}
\pm z=\begin{cases}
i \sqrt{L_\ell^2+L^2_{2,m}} \quad (m\in \Zb_{\geq 0}), &\text{if $L_\ell=L_{j,\ell_j}$ for a unique  $(j,\ell_j)=(1,\ell_1)$\,,}\\ 
i \sqrt{L_\ell^2+L^2_{1,m}} \quad (m\in \Zb_{\geq 0}), &\text{if $L_\ell=L_{j,\ell_j}$ for a unique $(j,\ell_j)=(2,\ell_2)$\,,}\\ 
i \sqrt{L_\ell^2+L^2_{1,m_1}} \quad \text{and} \quad i \sqrt{L_\ell^2+L^2_{2,m_2}}  &
(m_1,m_2\in \Zb_{\geq 0}), \\
 &\text{if $L_\ell=L_{1,\ell_1}=L_{2,\ell_2}$ for some $\ell_1,\ell_2$}\,.
\end{cases} 
\end{equation} 

\subsection{Extension across the negative imaginary axis}
\label{subsection:extIm<0}

The meromorphic continuation of $F$ across $-i(0,+\infty)$ will be done inductively, on certain Riemann surfaces constructed from the $\M_{\ell}$'s.  First, we will lift to these Riemann surfaces the local extensions of $F$ determined in Proposition \ref{pro:FextendedMainSection-Z}. Then, we will piece together these lifts along the branched curve lifting the interval $-i(0,L_N)$, where $N$ is a positive integer. We will be moving from branching point to branching point following the method developed in \cite[\S 3.4]{HPP14}. However, in contrast to the situation discussed there, our branching points are in general not evenly spaced. So additional care is needed. 

Let $N$ be a fixed positive integer and define
\begin{equation}
\label{eq:MN}
\M_{(N)}
=\left\{(z,\zeta)\in \C^- \times \C^{N+1} ; \zeta=(\zeta_0,\dots, \zeta_N),\ 
(z,\zeta_\ell)\in \M_\ell, \ \ell\in\Zb_{\geq 0}, \ 0\leq \ell\leq N\right\}.
\end{equation}
Then $\M_{(N)}$ is a Riemann surface, and the map
\begin{eqnarray}\label{finite cover 1}
&&\pi_{(N)}: \M_{(N)} \ni(z,\zeta)\to z\in \C^-
\end{eqnarray}
is a holomorphic $2^{N+1}$-to-$1$ cover, except when 
$z=-i L_\ell$ for some $\ell\in\Zb_{\geq 0}$ with $0\leq \ell\le N$. 
The fiber above each of these elements $-i L_\ell$ consists of $2^N$ branching points of $\M_{(N)}$. Recall the choice of square root function $\zeta_\ell^+(z)$ introduced in \eqref{eq:zeta+}. 
Making this choice for every coordinate function $\zeta_\ell$ on $\M_{(N)}$ yields
a section 
\[
\sigma_{(N)}^+: z\to (z,\zeta_0^+(z),\dots,\zeta_N^+(z))
\]
of the projection $\pi_{(N)}$. For $v \in \R^+$ we let 
$\zeta_\ell^+(-iv)=\zeta_\ell^+(-iv+0)$.
Then by \cite[Lemma 6]{HPP14} we have
\begin{equation}
\label{eq:zeta+Im}
\zeta_\ell^+(-iv)=\begin{cases}
\sqrt{\big(\frac{L_\ell}{v}\big)^2-1} &\text{if $0<v\leq L_\ell$}\\
i\sqrt{1-\big(\frac{L_\ell}{v}\big)^2} &\text{if $v\geq L_\ell$}\,.
\end{cases}
\end{equation}

We now construct a local meromorphic lift of $F$ to $\M_{(N)}$ by means of the right-hand side of \eqref{eq:Fresidues3} for every $m$ with $-1\leq m\leq N$. 
Observe that we are considering $\M_{(N)}$ as a Riemann surface above $\C^-$ to avoid the (known) logarithmic singularity of $F$ at $z=0$. 

For $0\leq \ell\leq N$ consider the holomorphic projection
\begin{eqnarray}\label{finite cover pi}
&&\pi_{(N,\ell)}: \M_{(N)} \ni(z,\zeta)\to (z,\zeta_\ell)\in \M_\ell.
\end{eqnarray}
Then the meromorphic function
\begin{equation}
\label{eq:wtGNell}
\wt G_{(N,\ell)}=\wt G_{(\ell)}\circ \pi_{(N,\ell)}: \M_{(N)}\to \C
\end{equation}
is holomorphic on ${(\pi_{(N)})}^{-1}(\C^-\setminus i\R)$. Moreover, on $\C^-\setminus i\R$,
\[
\wt G_{(N,\ell)}\circ  \sigma_{(N)}^+
=G_{(\ell)}
\]
by \eqref{eq:wtGsectionjl}. Therefore,  $\wt G_{(N,\ell)}$ is a meromorphic lift of $G_{(\ell)}$ to $\M_{(N)}$.

Using the right-hand side of \eqref{eq:Fresidues3} with $F_{(m)}$ constant on the $z$-fibers, we   obtain the requested lift of $F$ to $\pi_{(N)}^{-1}(W_{(m)}\setminus i\R)$.

The next step is to ``glue together'' all these local meromorphic extensions of $F$ to get a meromorphic extension of $F$ along the branched curve in $\M_{(N)}$ covering the interval $-i(0,L_{N+1})$.

An arbitrary section of  $\pi_{(N)}$ corresponds to a choice of sign $\pm\zeta_\ell^+$ for each coordinate function. More precisely,
with every 
\begin{equation}
\label{varepsilon}
\varepsilon=(\varepsilon_0,\dots,\varepsilon_N) \in \{\pm 1\}^{N+1}
\end{equation}
we associate the section 
\[
\sigma_\eps: \C^-\setminus
i(-\infty,-L] \to \M_{(N)}
\]
of $\pi_{(N)}$ given by 
\begin{equation}
\label{eq:sigmaeps}
\sigma_\eps(z)=(z,\varepsilon_0 \zeta_0^+(z),\dots,\varepsilon_N\zeta_N^+(z))\,.
\end{equation}
(Hence $\sigma_{(N)}^+$ corresponds to $\varepsilon=(1,\dots,1)$.)

To separate the study of the regular points from that of the branching points $-iL_{\ell}$, we exclude the latter from the open sets $W_{(m)}$ introduced in subsection \ref{piecewise}.  Hence, we define 
for $m\in \Zb_{\ge -1}$
\[
W'_{(m)}=\bigcup_{v\in (L_m,L_{m+1})} W_v\,
\]
where the $W_v$ are chosen as in \eqref{eq:WviR}.
Then $W_{(m)}=W_{L_m}\cup W'_{(m)}$ and $W'_{(m)}\cap i\R=-i(L_m,L_{m+1})$ contains 
none of the $-iL_\ell$.
By shrinking the disks $W_v$ if necessary, we may assume that $\pi_{(N)}^{-1}(W'_{(m)})$ is the disjoint union of $2^{N+1}$ homeomorphic copies of $W'_{(m)}$. In particular, each of these copies is a connected set. We denote by $U_{m,\eps}$ the copy containing
$\sigma_\eps(W'_{(m)}\setminus i\R)$. Likewise,  
we may assume that, for every $m \in \Zb_{\geq 0}$, the preimage under $\pi_{(N)}$ of the open neighborhood $W_{L_m}$ is the disjoint union of $2^{N}$ homeomorphic copies of the disk $W_{L_m}$. These homeomorphic copies can be parameterized by the elements
\eqref{varepsilon} with one component $\varepsilon_m$ removed. 
We denote by $\eps(m^\vee)$ the resulting element. Observe then that $\eps(m^\vee)=\eps'(m^\vee)$ if and only if $\eps$ and $\eps'$ are equal but for their $m$-th component, which can be $\pm 1$. Modulo this identification, we can indicate the connected components of $\pi_{(N)}^{-1}\big(W_{L_{m}}\big)$ by $U_{\eps(m^\vee)}$ with $\eps \in  \{\pm 1\}^{N+1}$. Furthermore, we define $U_{\eps(m^\vee)}=\emptyset$ for $m=-1$  and 
every $\eps\in \{\pm 1\}^{N+1}$.

We want to construct a meromorphic lift of $F$ along the branched curve
\begin{equation}
\label{eq:gamma}
\gamma_{N}: (0,L_{N+1})\ni v \to \big(-iv,\pm\zeta_0^+(-iv), \dots, \pm\zeta_N^+(-iv)\big) \in \M_{(N)}\,,
\end{equation}

Observe that $\big\{U_{\eps(m^\vee)}\cup U_{m,\eps}\ :\  \eps\in \{\pm 1\}^{N+1}, m\in \mathbb Z\,, -1\leq m\leq N\big\}$
is a covering of
$$\M_{\gamma_N}=\pi_{(N)}^{-1}\Big(\bigcup_{m=-1}^N W_{(m)}\Big)$$
consisting of open connected sets. Moreover, $\bigcup_{m=-1}^N W_{(m)}\cap i\R=-i(0,L_{N+1})$.

\begin{thm}
\label{thm:meroliftF}
For $m\in \{-1,0,\dots,N\}$, $\eps \in \{\pm 1\}^{N+1}$ and $(z,\zeta) \in U_{\eps(m^\vee)}\cup U_{m,\eps}$ define
\begin{equation}
\label{eq:widetildeF}
\wt F(z,\zeta)=
\displaystyle{F_{(m)}(z)+2 \sum_{\ell=0}^m\wt G_{(N,\ell)}(z,\zeta)+2
\sum_{\stackrel{m<\ell\leq N}{\text{with $\eps_\ell=-1$}}} \big[\wt G_{(N,\ell)}(z,\zeta)-
\wt G_{(N,\ell)}(z,-\zeta) \big]}\,,
 \end{equation}
where the first sum is equal to $0$ if $\ell=-1$ and the second sum is $0$ if $\eps_\ell=1$ for all $\ell>m$.
Then $\wt F$ is a meromorphic lift of $F$ to the open neighborhood
$\M_{\gamma_N}$ of the branched curve $\gamma_{N}$
lifting $-i(0,L_{N+1})$ in $\M_{(N)}$.

The singularities of $\wt F$ on $\M_{\gamma_N}$ are at most simple poles at the points
$(z,\zeta)\in \M_{(N)}$ for which there exists $(\ell_1,\ell_2)\in \mathbb{Z}_{\geq 0}$ such that 
$L_{1,\ell_1}^2+L_{2,\ell_2}^2<L_{N+1}^2$ and
\[
z=-i\sqrt{L_{1,\ell_1}^2+L_{2,\ell_2}^2}\,.
\]   
\end{thm}
\begin{proof}
Formula \eqref{eq:widetildeF} can be obtained as in the proof of \cite[Theorem 19]{HPP14}.  

Recall from \eqref{eq:sing-wtGell}  that the singularities of $\wt G_{(N,\ell)}(z,\zeta)=\wt G_{(\ell)}(z,\zeta_\ell)$ occur above values of $z\in i(-\infty,0)$ with $|z|> L_\ell$. Therefore the
second sum on the right-hand side of \eqref{eq:widetildeF} is holomorphic on $U_{\eps(m^\vee)}\cup U_{m,\eps}$ and the possible singularities of $\wt F$ on this domain come  from $\sum_{\ell=0}^m\wt G_{(N,\ell)}(z,\zeta)$. They are the elements $(z,\zeta) \in 
\M_{\gamma_N}$ with $z$ given by \eqref{eq:sing-wtGell}. The condition $L_{1,\ell_1}^2+L_{2,\ell_2}^2<L_{N+1}^2$ is necessary for the point $(z,\zeta)$ to be in $\M_{\gamma_N}$.
\end{proof}

To have a more precise description of the singularities of $\wt F$, let us order them according to their distance from the origin $0\in \C$. Let $(z_{(k)})_{k\in \Zb_{\geq 0}}$ be the sequence so obtained. For a fixed $k\in  \Zb_{\geq 0}$, let $S_{(1,k)}$ denote the set of all 
$(1,\ell_1)$ with $\ell_1\in  \Zb_{\geq 0}$ for which there exists $\ell_2\in \Zb_{\geq 0}$ so that $L_{1,\ell_1}^2+L_{2,\ell_2}^2=|z_{(k)}|^2$. The correspondence $(1,\ell_1) \to (2,\ell_2)$
so obtained is a bijection of  $S_{(1,k)}$  onto a subset $S_{(2,k)}$ of pairs
$(2,\ell_2)$ with $\ell_2 \in \Zb_{\geq 0}$. 

Let $N\in \Zb_{\geq 0}$ be chosen so that $|z_{(k)}|<L_{N+1}$. Then there is $m\in \{0,1,\dots, N\}$ such that $|z_{(k)}|\in [L_m,L_{m+1})$.
By \eqref{eq:widetildeF}, the possible singularities of $\wt F$ at points of $\M_{(N)}$ above $z_{(k)}$ are those of 
\[
{\sum}^{(k)}_\ell\; 
\wt G_{(N,\ell)}(z,\zeta)\,,
\]
where 
$\sum_\ell^{(k)}$ denotes the sum over all $\ell\in \{0,\dots,m\}$ for which there exists $(j,\ell_j)\in S_{(1,k)} \cup S_{(2,k)}$ so that $L_\ell=L_{j,\ell_j}$.
Observe that, by \eqref{eq:wtGNell} and \eqref{eq:wtGell}, this is equal to 
\begin{equation}
\label{wtFresidue1}
{\sum}^{(k)}_\ell\;
\wt G_{(\ell)}(z,\zeta_\ell)= 
\sum_{(j,\ell_j)\in  S_{(1,k)} \cup S_{(2,k)}}
\wt G_{j,\ell_j}(z,\zeta_{\ell_j})
= 
\sum_{(1,\ell_1)\in  S_{(1,k)}}
\Big(\wt G_{1,\ell_1}(z,\zeta_{\ell_1}) + \wt G_{2,\ell_2}(z,\zeta_{\ell_2}) \Big)
\,,
\end{equation}
where in each summand we choose the element $(2,\ell_2)\in  S_{(2,k)}$ associated with 
the fixed $(1,\ell_1)\in  S_{(1,k)}$.

Let $m\in \{0,1,\dots,N\}$ be fixed. For every $\ell \in \Zb_{\geq 0}$ with $0\leq \ell \leq m$ and  $\eps \in \{\pm 1\}^{N+1}$,  we consider the chart
\begin{equation}\label{chartmeps}
\kappa_{m,\eps,\ell}:U_{\eps(m^\vee)}\cup U_{m,\eps} \ni (z,\zeta) \to \zeta_\ell\in \C\setminus  i\big((-\infty, -1] \cup [1,+\infty)\big).
\end{equation}
Its inverse is given by 
\begin{equation}
(z,\zeta)=\kappa_{m,\eps,\ell}^{-1}(\zeta_\ell)
\quad \text{with} \quad z=-i\frac{L_\ell}{\sqrt{\zeta_\ell^2+1}}
\end{equation}
Notice that 
\begin{equation}\label{kappamepsell}
\kappa_{m,\eps,\ell}=\kappa_\ell \circ \pi_{(N,\ell)}|_{U_{\eps(m^\vee)}\cup U_{m,\eps} } 
\end{equation}
in the notation of page \pageref{notation:ell} and \eqref{finite cover pi}. Moreover, 
by \eqref{eq:sigmaeps}, \eqref{eq:zeta+Im} and the definition of $U_{\eps(m^\vee)}\cup U_{m,\eps}$, we have 
\[
\zeta_\ell=\kappa_{m,\eps,\ell} ((z,\zeta))=\eps_\ell \zeta^+_\ell(z)\,.
\]
We shall write $\kappa_{m,\eps}$ instead of $\kappa_{m,\eps,m}$. The local expression of $\wt F$  on $U_{\eps(m^\vee)}\cup U_{m,\eps}$ will be computed in terms of $\kappa_{m,\eps}$. However, to compute its residue at a singular point above $z_{(k)}$ with $|z_{(k)}|\in [L_m,L_{m+1})$ we will need to pass to some charts $\kappa_{m,\eps,\ell}$ with $0\leq \ell \leq m$, to use the results from Lemma \ref{lem:chart-expressions}. The following lemma 
examines the change of coordinates.

\begin{lem}
\label{lem:changecoord}
Suppose  $m\in \{0,1,\dots,N\}$, $|z_{(k)}|\in [L_m,L_{m+1})$ and $0\leq \ell \leq m$.
Then 
\begin{equation}
\label{relchartslm}
\wt G_{(N,\ell)} \circ \kappa_{m,\eps}^{-1}=
\Big(\wt G_{(\ell)} \circ \kappa_{\ell}^{-1}\big) \circ \big(\kappa_{m,\eps,\ell} \circ \kappa_{m,\eps}^{-1}\big)\,.
\end{equation}
The map $\kappa_{m,\eps,\ell} \circ \kappa_{m,\eps}^{-1}$ is a bijection of $\kappa_{m,\eps}(U_{\eps(m^\vee)}\cup U_{m,\eps} )$ onto 
$\kappa_{\ell}\big(\pi_{(N,\ell)}( (U_{\eps(m^\vee)}\cup U_{m,\eps} )\big)$. 
Let $\phi_{m,\eps,\ell}$ denote its inverse. Then
\begin{equation*}
\phi_{m,\eps,\ell} (\zeta_\ell)
=\eps_m \zeta_m^+ \big( -\tfrac{iL_\ell}{\sqrt{\zeta_\ell^2+1}}\big)\\
=\eps_m \sqrt{-\tfrac{L_m}{L_\ell} \, \sqrt{\zeta_\ell^2+1} +1}  \; 
\sqrt{-\tfrac{L_m}{L_\ell} \, \sqrt{\zeta_\ell^2+1} -1}\,.
\end{equation*}
Let $(z_{(k)},\zeta^{(k,\eps)})$ denote the point in $U_{\eps(m^\vee)}\cup U_{m,\eps}$ above $z_{(k)}$ and let $\zeta^{(k,\eps)}_\ell$ be the value of its $\zeta_\ell$-coordinate.
Then 
\begin{eqnarray}
\label{zetaellepsk}
\zeta^{(k,\eps)}_\ell&=&\eps_\ell \zeta_\ell^+(z_{(k)})= 
i\eps_\ell \;\frac{\sqrt{|z_{(k)}|^2-L_\ell^2}}{|z_{(k)}|}\,, \\
\label{phimepsellzeta}
\phi_{m,\eps,\ell} (\zeta^{(k,\eps)}_\ell)&=&\zeta^{(k,\eps)}_m
\end{eqnarray}
and
\begin{equation}
\label{dephiellmeps}
\phi_{m,\eps,\ell} '(\zeta^{(k,\eps)}_\ell)=\eps_\ell\eps_m \frac{L_m^2}{L_\ell^2} \; 
\frac{\sqrt{|z_{(k)}|^2-L_\ell^2}}{\sqrt{|z_{(k)}|^2-L_m^2}}\,.
\end{equation}
\end{lem}
\begin{proof}
The equality \eqref{relchartslm} is a consequence of \eqref{eq:wtGNell} and of the definition of $U_{\eps(m^\vee)}\cup U_{m,\eps}$, whereas \eqref{zetaellepsk} and \eqref{phimepsellzeta}
follow from \eqref{eq:zeta+Im}, since $iz_{(k)} =|z_{(k)}| \geq L_m\geq L_\ell$.
Notice that 
$\phi^2_{m,\eps,\ell} (\zeta_{\ell})=\frac{L_m^2}{L_\ell^2} \; (\zeta_\ell^2+1)-1$.
So $(\phi^2_{m,\eps,\ell}) '(\zeta_{\ell})=2\; \frac{L_m^2}{L_\ell^2} \; \zeta_\ell$. 
On the other hand, $(\phi^2_{m,\eps,\ell})' (\zeta_{\ell})=2(\phi_{m,\eps,\ell})' (\zeta_{\ell})
\phi_{m,\eps,\ell}(\zeta_{\ell})$.
Hence
\begin{equation*}
\phi_{m,\eps,\ell}' (\zeta^{(k,\eps)}_\ell)= \frac{(\phi^2_{m,\eps,\ell}) '(\zeta^{(k,\eps)}_\ell)}{2\phi_{m,\eps,\ell}(\zeta^{(k,\eps)}_\ell)}
=\frac{L_m^2}{L_\ell^2}  \; \frac{\zeta^{(k,\eps)}_\ell}{\phi_{m,\eps,\ell}(\zeta^{(k,\eps)}_\ell)}
=\frac{L_m^2}{L_\ell^2}  \; \eps_\ell\eps_m \frac{\sqrt{|z_{(k)}|^2-L_\ell^2}}{\sqrt{|z_{(k)}|^2-L_m^2}}
\end{equation*}
by \eqref{zetaellepsk} and \eqref{phimepsellzeta}.
\end{proof}

\begin{pro}
\label{pro:residueswtF}
Keep the above notation. 
Furthermore, set  for $k \in \Zb$,
\begin{equation}
\label{eq:Sk}
 S_{(k)}=\{(\ell_1,\ell_2)\in \Zb_{\geq 0}^2; \ L^2_{1,\ell_1}+L^2_{2,\ell_2}=|z_{(k)}|^2\}\,.
\end{equation}
and, for $\ell_1,\ell_2\in \Zb_{\geq 0}$, 
\begin{eqnarray}
\label{eq:Cepsell}
C_{\ell_1,\ell_2}&=&p_1(i L_{1,\ell_1})p_2(i L_{2,\ell_2})
 \big( b_1 \frac{L_{2,\ell_2}}{L_{1,\ell_1}}+  b_2\frac{L_{1,\ell_1}}{L_{2,\ell_2}}\big)\,, \\
\label{eq:lambdaell}
\lambda(\ell_1,\ell_2)&=&(\rho_{\beta_1}+\ell_1)\beta_1+(\rho_{\beta_2}+\ell_2)\beta_2\,.
\end{eqnarray}
Then $C_{\ell_1,\ell_2}$ is a positive constant, and
the residue of the local expression of $\wt F$ with respect to the 
chart $(U_{\eps(m^\vee)}\cup U_{m,\eps}, \kappa_{m,\eps})$ 
at the point of $\M_{(N)}$ above $z_{(k)}$ and in the domain of this chart  is  
\begin{equation}
\label{wtFresidue2}
\Res_{\zeta_m=\zeta^{(k,\eps)}_m} \big(\wt F \circ \kappa_{m,\eps}^{-1}\big)(\zeta_m) 
= \frac{i\eps_m L_m^2}{\pi^2 \sqrt{|z_{(k)}|^2-L_m^2}} \; 
\sum_{(\ell_1,\ell_2)\in  S_{(k)}} C_{\ell_1,\ell_2}
(f\times \varphi_{\lambda(\ell_1,\ell_2)})(y)\,.
\end{equation}
\end{pro}
\begin{proof}
By \eqref{wtFresidue1}, 
\[
\Res_{\zeta_m=\zeta^{(k,\eps)}_m} \big(\wt F \circ \kappa_{m,\eps}^{-1}\big)(\zeta_m)=
 {\sum}^{(k)}_\ell \Res_{\zeta_m=\zeta^{(k,\eps)}_m} \big( \wt G_{(N,\ell)}
\circ \kappa_{m,\eps}^{-1}\big)(\zeta_m)\,.
\]
Recall that if $f$ is meromorphic in the disk $\Dg_r(b)=\{z \in \C: |z-b|<r\}$ and $\varphi$ is a diffeomorphism defined on $\Dg_s(a)$ with $b=\varphi(a)$, then 
$\Res_{z=b} f(z)= \Res_{z=a} (f\circ \varphi)(z)\varphi'(z)$. See e.g. \cite[Ch. VI, Theorem 3.2]{Heins}. Hence, by Lemma \ref{lem:changecoord} and \eqref{eq:wtGNell}, 
\begin{eqnarray*}
\Res_{\zeta_m=\zeta^{(k,\eps)}_m} \big( \wt G_{(N,\ell)}
\circ \kappa_{m,\eps}^{-1}\big)(\zeta_m)
&=& \Res_{\zeta_\ell=\zeta^{(k,\eps)}_\ell} \big( \wt G_{(N,\ell)}
\circ \kappa_{m,\eps}^{-1} \circ \phi_{m,\eps,\ell}\big)(\zeta_\ell) \phi_{m,\eps,\ell}'(\zeta^{(k,\eps)}_\ell)\\
&=&  
\frac{L_m^2}{L_\ell^2}  \; \eps_\ell\eps_m \frac{\sqrt{|z_{(k)}|^2-L_\ell^2}}{\sqrt{|z_{(k)}|^2-L_m^2}}\; \Res_{\zeta_\ell=\zeta^{(k,\eps)}_\ell} \big( \wt G_{(\ell)}
\circ \kappa_\ell^{-1}\big)(\zeta_\ell)\,.
\end{eqnarray*}
Write $\zeta_{j,\ell_j,\eps,k}=\zeta^{(k,\eps)}_\ell$ and $\eps_{j,\ell_j}=\eps_\ell$ if 
$L_{j,\ell_j}=L_\ell$. Then, by \eqref{wtFresidue1} and the above computation, 
\begin{eqnarray*}
&&  {\sum}^{(k)}_\ell \Res_{\zeta_m=\zeta^{(k,\eps)}_m} \big( \wt G_{(N,\ell)}
\circ \kappa_{m,\eps}^{-1}\big)(\zeta_m)\\
&=&
\frac{\eps_m L_m^2}{\sqrt{|z_{(k)}|^2-L_m^2}} \;   
{\sum}^{(k)}_\ell \; \frac{\eps_\ell}{L_\ell^2} \sqrt{|z_{(k)}|^2-L_\ell^2} \; \Res_{\zeta_\ell=\zeta^{(k,\eps)}_\ell} \big( \wt G_{(\ell)}
\circ \kappa_\ell^{-1}\big)(\zeta_\ell)\\
&=&
\frac{\eps_m L_m^2}{\sqrt{|z_{(k)}|^2-L_m^2}} \; \sum_{(1,\ell_1)\in  S_{(1,k)}}
\sum_{j=1}^2
\Big[ \frac{\eps_{j,\ell_j}}{L_{j,\ell_j}^2} \sqrt{|z_{(k)}|^2-L_{j,\ell_j}^2} 
\Res_{\zeta_{\ell_j}=\zeta_{j,\ell_j,\eps,k}}
(\wt G_{j,\ell_j}\circ \kappa_{j,\ell_j,-}^{-1})(\zeta_{\ell_j}) 
\Big]
\end{eqnarray*}
The residues from Lemma \ref{lem:chart-expressions} and the fact that 
$|z_{(k)}|^2=L_{1,\ell_1}^2+L_{2,\ell_2}^2$ yield \eqref{wtFresidue2}. Notice that the sign 
of the residue at $\zeta_{j,\ell_j,\eps,k}$ cancels the factor $\eps_{j,\ell_j}$.
\end{proof}

\begin{rem}
Finding the list of the singularities $z_{(k)}=-i|z_{(k)}|$ of $\wt F$ and, for each of them,  determining the set $S_{(1,k)}$, is the problem of finding which positive reals can be written 
as a sum of squares $L_{1,\ell_1}^2 +L_{2,\ell_2}^2$ with $L_{j,\ell_j}$ in the lattices $b_j(\rho_{\beta_j}+\Zb_{\geq 0})$ ($j=1,2$), and in how many different ways. This is a variation of the classical problem of finding the positive integers which can be represented as a sum of squares of two integers and in how many ways. See e.g. \cite[section 6.7.4]{Lovett}.

Notice that the first singularity occurs at $z_{(0)}=-i\sqrt{L_{1,0}^2+L_{2,0}^2}=-i\sqrt{\langle \rho,\rho\rangle}$ and in this case $S_{(1,0)}$ contains only $(1,0)$.  
\end{rem}

%
%
%

\subsection{Meromorphic extension of the resolvent}
\label{subsection:meroextresolvent}

Recall that $\C^-=\{z\in \C;\Im z<0\}$ denotes the lower half plane, $L_{1,0}$ and $L_{2,0}$ are as in 
\eqref{eq:Lj}, and $L=\min\{L_{1,0},L_{2,0}\}$. In this subsection we meromorphically extend the resolvent
$z \mapsto [R(z)f](y)$ (where $f\in C_c^\infty(\X)$ and $y\in \X$ are arbitrarily fixed) from $\C^- \setminus i(-\infty,-L]$ across the half-line  $ i(-\infty,-L]$. As before, we shall omit the dependence on $f$ and $y$ from the notation and write $R(z)$ instead of $[R(z)f](y)$. This simplification of notation will be employed wherever it is appropriate.

The meromorphic extension of the resolvent will be deduced from that of $F$ obtained in the previous section. In fact, Proposition \ref{pro:holoextRF} shows that on an open subset of the upper half-plane we can write
\begin{equation}
\label{eq:holoextRF-mod}
R(z)=H(z)\,+\,\pi i\, F(z)\,,
\end{equation}
where $H$ is a holomorphic function on a domain containing $\C^-$. 

We keep the notation introduced in subsection \ref{subsection:extIm<0}. 

The resolvent $R$ can be lifted and meromorphically extended along the curve $\gamma_{N}$ in $\M_{\gamma_N}$. Its singularities 
(i.e. the resonances of the Laplacian) are those  of the meromorphic extension $\wt F$ of $F$ and are located at the  points 
of $\M_{\gamma_N}$ above the elements $z_{(k)}$. They are simple poles. The precise description is given by the following theorem.

\begin{thm}
\label{thm:meroextshiftedLaplacian}
Let $f \in C^\infty_c(\X)$ and $y \in \X$ be fixed. Let $N\in \mathbb N$ and let $\gamma_{N}$ be the curve in $\M_{(N)}$
given by (\ref{eq:gamma}). Then the resolvent $R(z)=[R(z)f](y)$ lifts as a meromorphic function to the neighborhood $\M_{\gamma_{N}}$ of the curve $\gamma_{N}$ in $\M_{(N)}$. We denote the lifted meromorphic function by
$\wt R_{(N)}(z,\zeta)=\big[\wt R_{(N)}(z,\zeta)f\big](y)$.

The singularities of $\wt R_{(N)}$ are at most simple poles at the points $(z_{(k)},\zeta^{(k,\eps)})\in \M_{(N)}$ with $k\in \Zb_{\geq 0}$ 
so that $|z_{(k)}|<L_{N+1}$ and $\eps \in\{\pm 1\}^{N+1}$. Explicitly, for $(m,\eps)\in \{0,1,\dots,N\}\times \{\pm 1\}^{N+1}$,
\begin{eqnarray}
\label{eq:RNnearwn}
\wt R_{(N)}(z,\zeta)=\wt H_{(N,m,\eps)}(z,\zeta)+2\pi i \sum_{\ell=0}^m \wt G_{(N,\ell)}(z,\zeta) \qquad
((z,\zeta)\in  U_{\eps(m^\vee)} \cup U_{m,\eps})
\,,
\end{eqnarray}
where $\wt H_{(N,m,\eps)}$ is holomorphic and $\wt G_{(N,\ell)}(z,\zeta)$
is in fact independent of $N$ and $\eps$ (but dependent on $f$ and $y$, which are omitted from the notation).
The singularities of $\wt R_{(N)}(z,\zeta)$ in $U_{\eps(m^\vee)} \cup U_{m,\eps}$ are simple poles at the points $(z_{(k)},\zeta^{(k,\eps)})$ belonging to $U_{\eps(m^\vee)} \cup U_{m,\eps}$ The residue of the local expression of $\wt R_{(N)}$ at this points is $i\pi$ times the right-hand side of \eqref{wtFresidue2}.
\end{thm}

\begin{rem}
Unlike the case of $\SL(3,\R)/\SO(3)$, the resonances usually do not coincide with the branching points of the Riemann surface $\M_{(N)}$ 
to which the resolvent lifts as a meromorphic function. 
In fact, one can check from the table in Remark \ref{rem:classification} that, when $\X_2=\X_1$ and $m_\beta$ is odd, none of the singularities 
happens to be at one of the branching  points $-iL_\ell$ unless $\X_1=\SU(2n,1)/\Sg(\Ug(2n)\times \Ug(1))$ (which is the only case where $\rho_\beta$ is an integer).
\end{rem}

\section{One odd multiplicity: holomorphic extensions}
\label{section:holoext-even}

In this section we show that when only one of the two multiplicities $m_{\beta_1}$ and $m_{\beta_2}$ is odd, then the resolvent of the Laplacian on $\X_1 \times \X_2$ admits a holomorphic extension to a suitable Riemann surface above $\C^\times$. As before, we assume 
without loss of generality that $m_{\beta_2}$ is even.

Fix $\ell\in\Bbb Z_{\geq 0}$. Then, as in \eqref{eq:Mjk}, we consider the Riemann surface
above $\C^\times$
\[
\M_{1,\ell}=\Big\{(z,\zeta)\in\C^\times\times(\C\setminus\{i, -i\})\ :\ \zeta^2=\Big(i\frac{L_{1,\ell}}{z}\Big)^2-1\Big\}
\]
with projection map $\pi_{1,\ell}:\M_{1,\ell} \ni (z,\zeta)\to z\in \C^\times$ and branching points 
$\pm (iL_{1,\ell},0)$.
Notice that Lemma \ref{2.8} holds in the present situation with $q_2=1$ and gives the lift $\wt G_{1,\ell}$ of $G_{1,l}$ to $\M_{1,\ell}$. Since now $q_2$ is holomorphic, $\wt G_{1,\ell}$ turns out to be holomorphic as well.
\begin{lem}
\label{lemma:holo-wtG-even}
Suppose that $m_{\beta_1}$ is odd and $m_{\beta_2}$ is even. Then, in the notation of subsection
\ref{subsection:contours-even-mul},
\begin{eqnarray}\label{eq:tG1k-even}
\wt G_{1,\ell}:\M_{1,\ell}\ni (z,\zeta)&\to&
\frac{b_1}{\pi} \, L_{1,\ell}\,  p_1(i L_{1,\ell})p_2(-z\zeta)  \psi_z\big(\frac{i L_{1,\ell}}{z}-\zeta\big) \in\C
\end{eqnarray}
is the holomorphic extension to $\M_{1,\ell}$ of a lift of $G_{1,\ell}$.
\end{lem}

For the piecewise extension of $F$ along the negative imaginary half-line $-i[b_1\rho_{\beta_1},+\infty)$, recall from Proposition \ref{pro:local-ext-F-even} that for $v\in I_{1,m}=b_1\rho_{\beta_1}  +b_1[m_1,m_1+1)=[L_{j,m},L_{j,m+1})$ with $m\in \Zb_{\geq 0}$ there exists 
$0< r_v <1$ and an open neighborhood $W_v$ of $-iv$ in $\C$ such that
\begin{equation}
\label{eq:Fresidues1-even}
F(z)=F_{r_v}(z)+ 2\sum_{\ell=0}^{m} G_{1,\ell}(z) 
\qquad (z\in W_v\setminus i\R)\, ,
\end{equation}
where the function $F_{r_v}$ is holomorphic in $W_v$. 
As before, \eqref{eq:Fresidues1-even} extends to $I_{1,-1}=(0,L_{1,0})$ by allowing empty sums.
By possibly shrinking $W_v$, we may assume that $W_v$ is an open disk around $-iv$ such that
$$
W_v\cap i\R \subseteq
\begin{cases}
  -iI_{1,m}&\text{ for } v\in I_{1,m}^\circ,\\
  -i(I_{1,m}-\frac{b_1}{2})&\text{ for } v=L_{1,m}\,.
\end{cases}
$$
In addition, for $0<v < b_1\rho_{\beta_1}$ we define $W_v$ to be an open ball around $-iv$ in $\C$ such that $W_{v}\cap i\R\subset (0,b_1\rho_{\beta_1})$.

If $v\in I_{1,m}$, $v'\geq b_1\rho_{\beta_1}$ and $W_v\cap W_{v'}\not=\emptyset$, then we obtain for $z\in W_v\cap W_{v'}$
$$
F_{r_{v'}}(z)=F_{r_v}(z)+
\begin{cases}
  0&\text{ for }   v'\in I_{1,m},\\
  2\, G_{1,m_1}(z)&\text{ for } v'\in I_{1,m_1-1}. \\
\end{cases}
$$
Now we set
\begin{eqnarray*}
W_{(m)}&=&\bigcup_{v\in I_{1,m}} W_v \qquad (m\in \Zb_{\geq 0})\,,\\
W_{(-1)}&=&\bigcup_{v\in I_{1,-1}} W_v \,.
\end{eqnarray*}
For $m\in\Zb_{\geq 1}$ we define a holomorphic function $F_{(m)}: W_{(m)}\to \C$ by 
\[
F_{(m)}(z)= \begin{cases}
F_{r_v}(z) &\text{for $m\in \Zb_{\geq 0}$, $v\in I_{1,m}$ and $z\in W_v$}\\
F(z) &\text{for $m=-1$ and $z\in W_{(-1)}$}\,.
\end{cases}
\]
We therefore obtain the following analogue of Proposition \ref{cor:FextendedMainSection}.

\begin{pro}
\label{cor:FextendedMainSection-even}
For every integers $(m_1,m_2)\in \Zb_{\ge -1}^2$ we have
 \begin{equation}
\label{eq:Fresidues2-even}
F(z)=F_{(m)}(z)+2\sum_{\ell_1=0}^{m} G_{1,\ell}(z)
\qquad (z\in W_{(m)}\setminus i\R)\,,
\end{equation}
where $F_{(m)}$ is holomorphic in $W_{(m)}$, the $G_{1,\ell}$ are as in \eqref{eq:G1l-even}, and empty sums are defined to be equal to $0$.
\end{pro}

We can continue $F$ across $-i[0,+\infty)$ inductively, as in the case of two odd multiplicities. The situation is however easier, either because there is only one regularly spaced sequence of branching points (there are no branching points of the form $-iL_{2,\ell}$, since they originated from the singularities of the Plancherel density of $\X_2$), or because the continuation turns out to be holomorphic (because of Lemma \ref{lemma:holo-wtG-even}).

Let $N$ be a fixed positive integer and let $\M_{(N)}$ be defined by \eqref{eq:MN} with $\M_\ell$ replaced by $\M_{1,\ell}$ for $0\leq \ell\leq N$.  Then $\M_{(N)}$  is a Riemann surface which is a $2^{N+1}$-to-$1$ cover of $\C^-$, except when $z=-i L_{1,\ell}$ with  $\ell\in\Zb_{\geq 0}$,  $0\leq \ell\le N$. The fiber of the latter points consists of $2^N$ branching points of $\M_{(N)}$.  

For $0\leq \ell\leq N$ define 
\[
\wt G_{(N,\ell)}=\wt G_{1,\ell} \circ \pi_{N,\ell}\,,
\]
where $\pi_{N,\ell}: \M_{(N)} \ni(z,\zeta)\to (z,\zeta_\ell) \in \M_{1,\ell}$.  Furthermore, define as in section \ref{subsection:extIm<0}, the open sets $U_{m,\eps}$, $U_{\eps(m^\vee)}$  (with $m\in \Zb_{\geq 0}$, $\eps\in \{\pm 1\}^{N+1}$), the branched curve $\gamma_N$  lifting $-i(0,L_{1,N+1})$ to $\M_{(N)}$, and the open neighborhood $\M_{\gamma_N}$ of $\gamma_N$ in 
$\M_{(N)}$. Then we have the following analogues of Theorems \ref{thm:meroliftF} and \ref{thm:meroextshiftedLaplacian}.
\begin{thm}
\label{thm:holoext-even}
Let $f\in C^\infty(\X)$ and $y \in \X$ be fixed. Let $N$ be a positive integer. 
Then the function $F$ from Proposition \ref{pro:local-ext-F-even} (which depends on $f$ and $y$) admits a holomorphic extension $\wt F$ to $\M_{\gamma_N}$ given by \eqref{eq:widetildeF}.
The resolvent $R(z)=[R(z)f](y)$ of the Laplacian lifts and extends holomorphically to the function 
$\wt R_{(N)}$ given on $\M_{\gamma_N}$ by the formula 
\[
\wt R_{(N)}(z,\zeta)=\wt H(z,\zeta)+i \pi\wt F(z,\zeta)\,,
\]
where $\wt H(z,\zeta)=H(z)$ is a lift of the holomorphic function $H$ of Proposition \ref{pro:FextendedMainSection-Z}. Consequently, the Laplacian has no resonances.
\end{thm}

\section{The residue operators}
\label{section:residueops}

In this section we come back to the case where both multiplicities $m_{\beta_1}$ and $m_{\beta_2}$ are odd. By Theorem \ref{thm:meroextshiftedLaplacian}, for every positive integer $N$, the resolvent of the (shifted) Laplacian of $\X=\X_1\times \X_2$ admits a meromorphic continuation $\wt R_{(N)}$ on suitable Riemann surface $\M_{(N)}$. The extended resolvent has simple poles (i.e. resonances) at the points of  $\M_{(N)}$ above the 
negative purely imaginary values $z_{(k)}$ ($k\in \Zb_{\geq 0}$) introduced in section \ref{subsection:extIm<0}. 

Associating with $f\in C^\infty(\X)$ the residue at the point $(z_{(k)},\zeta^{(\eps,k}))$ of the extended resolvent $\wt R_{(N)}(z)f$ yields a continuous linear operator
\[
{\Res}_{N,k} : C^\infty_c(\X) \to C^\infty(\X)\,, 
\]
where 
\begin{equation}
\label{eq:Reskeps}
{\Res}_{N,k} f= \sum_{(\ell_1,\ell_2)\in  S_{(k)}} C_{\ell_1,\ell_2}
\; (f\times \varphi_{\lambda(\ell_1,\ell_2)}) \qquad
(f\in C^\infty_c(\X))
\end{equation}
and $S_{(k)}$, $C_{\ell_1,\ell_2}$, $\varphi_{\lambda(\ell_1,\ell_2)}$ are as in 
\eqref{eq:Sk}, \eqref{eq:lambdaell} and \eqref{eq:Cepsell}, respectively. 
 Notice that ${\Res}_{N,k}$ is independent of the choice of the point in the fiber of $z_{(k)}$.


The element $\lambda(\ell_1,\ell_2)$ given by \eqref{eq:lambdaell} is a highest restricted weight of $\X$. Indeed, for $j=1,2$ we have $(\lambda(\ell_1,\ell_2)-\rho)_{\beta_j}=\ell_j\in \Zb_{\geq 0}$. Recall the operator $\mathcal{R}_\lambda$ from \eqref{eq:Rlambda}.
By the properties of the eigenspace representations given in subsection \ref{subsection:X},
\[
\mathcal{R}_{\lambda(\ell_1,\ell_2)}\big(C^\infty_c(\X)\big)=\mathcal{E}_{\lambda(\ell_1,\ell_2),\G}(\X)
=\mathcal{E}_{(\rho_{\beta_1}+\ell_1)\beta_1,\G}(\X_1)\otimes \mathcal{E}_{(\rho_{\beta_2}+\ell_2)\beta_2,\G}(\X_2)
\]
is the finite dimensional irreducible spherical representation of $\G$ of highest restricted weight $\lambda(\ell_1,\ell_2)$.

We  therefore obtain the following proposition.

\begin{pro}\label{pro:rank}
Let $N$ be a positive integer and let $k\in \Zb_{\geq 0}$ be so that $|z_{(k)}|<L_{N+1}^2$. 
Then the (resolvent) residue operator $R_{N,k}$ at $(z_{(k)},\zeta^{(\eps,k)})\in \M_{(N)}$ has image
\begin{eqnarray*}
{\Res}_{N,k}\big(C^\infty_c(\X)\big)&=&
\oplus_{(\ell_1,\ell_2)\in S_{(k)}} 
\mathcal{E}_{\lambda(\ell_1,\ell_2),\G}(\X)\\
&=&\oplus_{(\ell_1,\ell_2)\in S_{(k)}}  \big(\mathcal{E}_{(\rho_{\beta_1}+\ell_1)\beta_1,\G}(\X_1)\otimes \mathcal{E}_{(\rho_{\beta_2}+\ell_2)\beta_2,\G}(\X_2)\big)\,.
\end{eqnarray*}
In particular, ${\Res}_{N,k}$ has finite rank and is independent of the choice of the point $(z_{(k)},\zeta^{(\eps,k)})$ in the fiber of $z_{(k)}$ in $\M_{(N)}$.
\end{pro}

\end{document}